\documentclass[11pt]{article}
\usepackage{geometry}                
\usepackage[utf8]{inputenc}
\usepackage{amsmath,amsfonts,amsthm,amssymb}

\DeclareUnicodeCharacter{00A0}{ }

\usepackage{tikz, tikz-3dplot}

\definecolor{airforceblue}{rgb}{0.36, 0.54, 0.66}			    
\definecolor{aliceblue}{rgb}{0.94, 0.97, 1.0}				    
\definecolor{alizarin}{rgb}{0.82, 0.1, 0.26}				    
\definecolor{almond}{rgb}{0.94, 0.87, 0.8}					    
\definecolor{amaranth}{rgb}{0.9, 0.17, 0.31}				    
\definecolor{amber}{rgb}{1.0, 0.75, 0.0}						
\definecolor{amber(sae/ece)}{rgb}{1.0, 0.49, 0.0}	            
\definecolor{americanrose}{rgb}{1.0, 0.01, 0.24}	            
\definecolor{amethyst}{rgb}{0.6, 0.4, 0.8}	                    
\definecolor{anti-flashwhite}{rgb}{0.95, 0.95, 0.96}	        
\definecolor{antiquebrass}{rgb}{0.8, 0.58, 0.46}	            
\definecolor{antiquefuchsia}{rgb}{0.57, 0.36, 0.51}	            
\definecolor{antiquewhite}{rgb}{0.98, 0.92, 0.84}	            
\definecolor{ao}{rgb}{0.0, 0.0, 1.0}	                        
\definecolor{ao(english)}{rgb}{0.0, 0.5, 0.0}	                
\definecolor{applegreen}{rgb}{0.55, 0.71, 0.0}	                
\definecolor{apricot}{rgb}{0.98, 0.81, 0.69}	                
\definecolor{aqua}{rgb}{0.0, 1.0, 1.0}	                        
\definecolor{aquamarine}{rgb}{0.5, 1.0, 0.83}	                
\definecolor{armygreen}{rgb}{0.29, 0.33, 0.13}	                
\definecolor{arsenic}{rgb}{0.23, 0.27, 0.29}	                
\definecolor{arylideyellow}{rgb}{0.91, 0.84, 0.42}	            
\definecolor{ashgrey}{rgb}{0.7, 0.75, 0.71}	                    
\definecolor{asparagus}{rgb}{0.53, 0.66, 0.42}	                
\definecolor{atomictangerine}{rgb}{1.0, 0.6, 0.4}	            
\definecolor{auburn}{rgb}{0.43, 0.21, 0.1}	                    
\definecolor{aureolin}{rgb}{0.99, 0.93, 0.0}	                
\definecolor{aurometalsaurus}{rgb}{0.43, 0.5, 0.5}	            
\definecolor{awesome}{rgb}{1.0, 0.13, 0.32}	                    
\definecolor{azure(colorwheel)}{rgb}{0.0, 0.5, 1.0}	            
\definecolor{azure(web)(azuremist)}{rgb}{0.94, 1.0, 1.0}	    
\definecolor{babyblue}{rgb}{0.54, 0.81, 0.94}	                
\definecolor{babyblueeyes}{rgb}{0.63, 0.79, 0.95}	            
\definecolor{babypink}{rgb}{0.96, 0.76, 0.76}	                
\definecolor{ballblue}{rgb}{0.13, 0.67, 0.8}	                
\definecolor{bananamania}{rgb}{0.98, 0.91, 0.71}	            
\definecolor{bananayellow}{rgb}{1.0, 0.88, 0.21}	            
\definecolor{battleshipgrey}{rgb}{0.52, 0.52, 0.51}	            
\definecolor{bazaar}{rgb}{0.6, 0.47, 0.48}	                    
\definecolor{beaublue}{rgb}{0.74, 0.83, 0.9}	                
\definecolor{beaver}{rgb}{0.62, 0.51, 0.44}	                    
\definecolor{beige}{rgb}{0.96, 0.96, 0.86}	                    
\definecolor{bisque}{rgb}{1.0, 0.89, 0.77}	                    
\definecolor{bistre}{rgb}{0.24, 0.17, 0.12}	                    
\definecolor{bittersweet}{rgb}{1.0, 0.44, 0.37}	                
\definecolor{black}{rgb}{0.0, 0.0, 0.0}	                        
\definecolor{blanchedalmond}{rgb}{1.0, 0.92, 0.8}	            
\definecolor{bleudefrance}{rgb}{0.19, 0.55, 0.91}	            
\definecolor{blizzardblue}{rgb}{0.67, 0.9, 0.93}	            
\definecolor{blond}{rgb}{0.98, 0.94, 0.75}	                    
\definecolor{blue}{rgb}{0.0, 0.0, 1.0}	                        
\definecolor{blue(munsell)}{rgb}{0.0, 0.5, 0.69}	            
\definecolor{blue(ncs)}{rgb}{0.0, 0.53, 0.74}	                
\definecolor{blue(pigment)}{rgb}{0.2, 0.2, 0.6}	                
\definecolor{blue(ryb)}{rgb}{0.01, 0.28, 1.0}	                
\definecolor{bluebell}{rgb}{0.64, 0.64, 0.82}	                
\definecolor{bluegray}{rgb}{0.4, 0.6, 0.8}	                    
\definecolor{blue-green}{rgb}{0.0, 0.87, 0.87}	                
\definecolor{blue-violet}{rgb}{0.54, 0.17, 0.89}	            
\definecolor{blush}{rgb}{0.87, 0.36, 0.51}	                    
\definecolor{bole}{rgb}{0.47, 0.27, 0.23}	                    
\definecolor{bondiblue}{rgb}{0.0, 0.58, 0.71}	                
\definecolor{bostonuniversityred}{rgb}{0.8, 0.0, 0.0}	        
\definecolor{brandeisblue}{rgb}{0.0, 0.44, 1.0}	                
\definecolor{brass}{rgb}{0.71, 0.65, 0.26}	                    
\definecolor{brickred}{rgb}{0.8, 0.25, 0.33}	                
\definecolor{brightcerulean}{rgb}{0.11, 0.67, 0.84}	            
\definecolor{brightgreen}{rgb}{0.4, 1.0, 0.0}	                
\definecolor{brightlavender}{rgb}{0.75, 0.58, 0.89}	            
\definecolor{brightmaroon}{rgb}{0.76, 0.13, 0.28}	            
\definecolor{brightpink}{rgb}{1.0, 0.0, 0.5}	                
\definecolor{brightturquoise}{rgb}{0.03, 0.91, 0.87}	        
\definecolor{brightube}{rgb}{0.82, 0.62, 0.91}	                
\definecolor{brilliantlavender}{rgb}{0.96, 0.73, 1.0}	        
\definecolor{brilliantrose}{rgb}{1.0, 0.33, 0.64}	            
\definecolor{brinkpink}{rgb}{0.98, 0.38, 0.5}	                
\definecolor{britishracinggreen}{rgb}{0.0, 0.26, 0.15}	        
\definecolor{bronze}{rgb}{0.8, 0.5, 0.2}	                    
\definecolor{brown(traditional)}{rgb}{0.59, 0.29, 0.0}	        
\definecolor{brown(web)}{rgb}{0.65, 0.16, 0.16}	                
\definecolor{bubblegum}{rgb}{0.99, 0.76, 0.8}	                
\definecolor{bubbles}{rgb}{0.91, 1.0, 1.0}	                    
\definecolor{buff}{rgb}{0.94, 0.86, 0.51}	                    
\definecolor{bulgarianrose}{rgb}{0.28, 0.02, 0.03}	            
\definecolor{burgundy}{rgb}{0.5, 0.0, 0.13}	                    
\definecolor{burlywood}{rgb}{0.87, 0.72, 0.53}	                
\definecolor{burntorange}{rgb}{0.8, 0.33, 0.0}	                
\definecolor{burntsienna}{rgb}{0.91, 0.45, 0.32}	            
\definecolor{burntumber}{rgb}{0.54, 0.2, 0.14}	                
\definecolor{byzantine}{rgb}{0.74, 0.2, 0.64}	                
\definecolor{byzantium}{rgb}{0.44, 0.16, 0.39}	                
\definecolor{cadet}{rgb}{0.33, 0.41, 0.47}	                    
\definecolor{cadetblue}{rgb}{0.37, 0.62, 0.63}	                
\definecolor{cadetgrey}{rgb}{0.57, 0.64, 0.69}	                
\definecolor{cadmiumgreen}{rgb}{0.0, 0.42, 0.24}	            
\definecolor{cadmiumorange}{rgb}{0.93, 0.53, 0.18}	            
\definecolor{cadmiumred}{rgb}{0.89, 0.0, 0.13}	                
\definecolor{cadmiumyellow}{rgb}{1.0, 0.96, 0.0}	            
\definecolor{calpolypomonagreen}{rgb}{0.12, 0.3, 0.17}	        
\definecolor{cambridgeblue}{rgb}{0.64, 0.76, 0.68}	            
\definecolor{camel}{rgb}{0.76, 0.6, 0.42}	                    
\definecolor{camouflagegreen}{rgb}{0.47, 0.53, 0.42}	        
\definecolor{canaryyellow}{rgb}{1.0, 0.94, 0.0}	                
\definecolor{candyapplered}{rgb}{1.0, 0.03, 0.0}	            
\definecolor{candypink}{rgb}{0.89, 0.44, 0.48}	                
\definecolor{capri}{rgb}{0.0, 0.75, 1.0}	                    
\definecolor{caputmortuum}{rgb}{0.35, 0.15, 0.13}	            
\definecolor{cardinal}{rgb}{0.77, 0.12, 0.23}	                
\definecolor{caribbeangreen}{rgb}{0.0, 0.8, 0.6}	            
\definecolor{carmine}{rgb}{0.59, 0.0, 0.09}	                    
\definecolor{carminepink}{rgb}{0.92, 0.3, 0.26}	                
\definecolor{carminered}{rgb}{1.0, 0.0, 0.22}	                
\definecolor{carnationpink}{rgb}{1.0, 0.65, 0.79}	            
\definecolor{carnelian}{rgb}{0.7, 0.11, 0.11}	                
\definecolor{carolinablue}{rgb}{0.6, 0.73, 0.89}	            
\definecolor{carrotorange}{rgb}{0.93, 0.57, 0.13}	            
\definecolor{ceil}{rgb}{0.57, 0.63, 0.81}	                    
\definecolor{celadon}{rgb}{0.67, 0.88, 0.69}	                
\definecolor{celestialblue}{rgb}{0.29, 0.59, 0.82}	            
\definecolor{cerise}{rgb}{0.87, 0.19, 0.39}	                    
\definecolor{cerisepink}{rgb}{0.93, 0.23, 0.51}	                
\definecolor{cerulean}{rgb}{0.0, 0.48, 0.65}	                
\definecolor{ceruleanblue}{rgb}{0.16, 0.32, 0.75}	            
\definecolor{chamoisee}{rgb}{0.63, 0.47, 0.35}	                
\definecolor{champagne}{rgb}{0.97, 0.91, 0.81}	                
\definecolor{charcoal}{rgb}{0.21, 0.27, 0.31}	                
\definecolor{chartreuse(traditional)}{rgb}{0.87, 1.0, 0.0}	    
\definecolor{chartreuse(web)}{rgb}{0.5, 1.0, 0.0}	            
\definecolor{cherryblossompink}{rgb}{1.0, 0.72, 0.77}	        
\definecolor{chestnut}{rgb}{0.8, 0.36, 0.36}	                
\definecolor{chocolate(traditional)}{rgb}{0.48, 0.25, 0.0}	    
\definecolor{chocolate(web)}{rgb}{0.82, 0.41, 0.12}	            
\definecolor{chromeyellow}{rgb}{1.0, 0.65, 0.0}	                
\definecolor{cinereous}{rgb}{0.6, 0.51, 0.48}	                
\definecolor{cinnabar}{rgb}{0.89, 0.26, 0.2}	                
\definecolor{cinnamon}{rgb}{0.82, 0.41, 0.12}	                
\definecolor{citrine}{rgb}{0.89, 0.82, 0.04}	                
\definecolor{classicrose}{rgb}{0.98, 0.8, 0.91}	                
\definecolor{cobalt}{rgb}{0.0, 0.28, 0.67}	                    
\definecolor{cocoabrown}{rgb}{0.82, 0.41, 0.12}	                
\definecolor{columbiablue}{rgb}{0.61, 0.87, 1.0}	            
\definecolor{coolblack}{rgb}{0.0, 0.18, 0.39}	                
\definecolor{coolgrey}{rgb}{0.55, 0.57, 0.67}	                
\definecolor{copper}{rgb}{0.72, 0.45, 0.2}	                    
\definecolor{copperrose}{rgb}{0.6, 0.4, 0.4}	                
\definecolor{coquelicot}{rgb}{1.0, 0.22, 0.0}	                
\definecolor{coral}{rgb}{1.0, 0.5, 0.31}	                    
\definecolor{coralpink}{rgb}{0.97, 0.51, 0.47}	                
\definecolor{coralred}{rgb}{1.0, 0.25, 0.25}	                
\definecolor{cordovan}{rgb}{0.54, 0.25, 0.27}	                
\definecolor{corn}{rgb}{0.98, 0.93, 0.36}	                    
\definecolor{cornellred}{rgb}{0.7, 0.11, 0.11}	                
\definecolor{cornflowerblue}{rgb}{0.39, 0.58, 0.93}	            
\definecolor{cornsilk}{rgb}{1.0, 0.97, 0.86}	                
\definecolor{cosmiclatte}{rgb}{1.0, 0.97, 0.91}	                
\definecolor{cottoncandy}{rgb}{1.0, 0.74, 0.85}	                
\definecolor{cream}{rgb}{1.0, 0.99, 0.82}	                    
\definecolor{crimson}{rgb}{0.86, 0.08, 0.24}	                
\definecolor{crimsonglory}{rgb}{0.75, 0.0, 0.2}	                
\definecolor{cyan}{rgb}{0.0, 1.0, 1.0}	                        
\definecolor{cyan(process)}{rgb}{0.0, 0.72, 0.92}	            
\definecolor{daffodil}{rgb}{1.0, 1.0, 0.19}	                    
\definecolor{dandelion}{rgb}{0.94, 0.88, 0.19}	                
\definecolor{darkblue}{rgb}{0.0, 0.0, 0.55}	                    
\definecolor{darkbrown}{rgb}{0.4, 0.26, 0.13}	                
\definecolor{darkbyzantium}{rgb}{0.36, 0.22, 0.33}	            
\definecolor{darkcandyapplered}{rgb}{0.64, 0.0, 0.0}	        
\definecolor{darkcerulean}{rgb}{0.03, 0.27, 0.49}	            
\definecolor{darkchampagne}{rgb}{0.76, 0.7, 0.5}	            
\definecolor{darkchestnut}{rgb}{0.6, 0.41, 0.38}	            
\definecolor{darkcoral}{rgb}{0.8, 0.36, 0.27}	                
\definecolor{darkcyan}{rgb}{0.0, 0.55, 0.55}	                
\definecolor{darkelectricblue}{rgb}{0.33, 0.41, 0.47}	        
\definecolor{darkgoldenrod}{rgb}{0.72, 0.53, 0.04}	            
\definecolor{darkgray}{rgb}{0.66, 0.66, 0.66}	                
\definecolor{darkgreen}{rgb}{0.0, 0.2, 0.13}	                
\definecolor{darkjunglegreen}{rgb}{0.1, 0.14, 0.13}	            
\definecolor{darkkhaki}{rgb}{0.74, 0.72, 0.42}	                
\definecolor{darklava}{rgb}{0.28, 0.24, 0.2}	                
\definecolor{darklavender}{rgb}{0.45, 0.31, 0.59}	            
\definecolor{darkmagenta}{rgb}{0.55, 0.0, 0.55}	                
\definecolor{darkmidnightblue}{rgb}{0.0, 0.2, 0.4}	            
\definecolor{darkolivegreen}{rgb}{0.33, 0.42, 0.18}	            
\definecolor{darkorange}{rgb}{1.0, 0.55, 0.0}	                
\definecolor{darkorchid}{rgb}{0.6, 0.2, 0.8}	                
\definecolor{darkpastelblue}{rgb}{0.47, 0.62, 0.8}	            
\definecolor{darkpastelgreen}{rgb}{0.01, 0.75, 0.24}	        
\definecolor{darkpastelpurple}{rgb}{0.59, 0.44, 0.84}	        
\definecolor{darkpastelred}{rgb}{0.76, 0.23, 0.13}	            
\definecolor{darkpink}{rgb}{0.91, 0.33, 0.5}	                
\definecolor{darkpowderblue}{rgb}{0.0, 0.2, 0.6}	            
\definecolor{darkraspberry}{rgb}{0.53, 0.15, 0.34}	            
\definecolor{darkred}{rgb}{0.55, 0.0, 0.0}	                    
\definecolor{darksalmon}{rgb}{0.91, 0.59, 0.48}	                
\definecolor{darkscarlet}{rgb}{0.34, 0.01, 0.1}	                
\definecolor{darkseagreen}{rgb}{0.56, 0.74, 0.56}	            
\definecolor{darksienna}{rgb}{0.24, 0.08, 0.08}	                
\definecolor{darkslateblue}{rgb}{0.28, 0.24, 0.55}	            
\definecolor{darkslategray}{rgb}{0.18, 0.31, 0.31}	            
\definecolor{darkspringgreen}{rgb}{0.09, 0.45, 0.27}	        
\definecolor{darktan}{rgb}{0.57, 0.51, 0.32}	                
\definecolor{darktangerine}{rgb}{1.0, 0.66, 0.07}	            
\definecolor{darktaupe}{rgb}{0.28, 0.24, 0.2}	                
\definecolor{darkterracotta}{rgb}{0.8, 0.31, 0.36}	            
\definecolor{darkturquoise}{rgb}{0.0, 0.81, 0.82}	            
\definecolor{darkviolet}{rgb}{0.58, 0.0, 0.83}	                
\definecolor{dartmouthgreen}{rgb}{0.05, 0.5, 0.06}	            
\definecolor{davy\'sgrey}{rgb}{0.33, 0.33, 0.33}	            
\definecolor{debianred}{rgb}{0.84, 0.04, 0.33}	                
\definecolor{deepcarmine}{rgb}{0.66, 0.13, 0.24}	            
\definecolor{deepcarminepink}{rgb}{0.94, 0.19, 0.22}	        
\definecolor{deepcarrotorange}{rgb}{0.91, 0.41, 0.17}	        
\definecolor{deepcerise}{rgb}{0.85, 0.2, 0.53}	                
\definecolor{deepchampagne}{rgb}{0.98, 0.84, 0.65}	            
\definecolor{deepchestnut}{rgb}{0.73, 0.31, 0.28}	            
\definecolor{deepfuchsia}{rgb}{0.76, 0.33, 0.76}	            
\definecolor{deepjunglegreen}{rgb}{0.0, 0.29, 0.29}	            
\definecolor{deeplilac}{rgb}{0.6, 0.33, 0.73}	                
\definecolor{deepmagenta}{rgb}{0.8, 0.0, 0.8}	                
\definecolor{deeppeach}{rgb}{1.0, 0.8, 0.64}	                
\definecolor{deeppink}{rgb}{1.0, 0.08, 0.58}	                
\definecolor{deepsaffron}{rgb}{1.0, 0.6, 0.2}	                
\definecolor{deepskyblue}{rgb}{0.0, 0.75, 1.0}	                
\definecolor{denim}{rgb}{0.08, 0.38, 0.74}	                    
\definecolor{desert}{rgb}{0.76, 0.6, 0.42}	                    
\definecolor{desertsand}{rgb}{0.93, 0.79, 0.69}	                
\definecolor{dimgray}{rgb}{0.41, 0.41, 0.41}	                
\definecolor{dodgerblue}{rgb}{0.12, 0.56, 1.0}	                
\definecolor{dogwoodrose}{rgb}{0.84, 0.09, 0.41}	            
\definecolor{dollarbill}{rgb}{0.52, 0.73, 0.4}	                
\definecolor{drab}{rgb}{0.59, 0.44, 0.09}	                    
\definecolor{dukeblue}{rgb}{0.0, 0.0, 0.61}	                    
\definecolor{earthyellow}{rgb}{0.88, 0.66, 0.37}	            
\definecolor{ecru}{rgb}{0.76, 0.7, 0.5}	                        
\definecolor{eggplant}{rgb}{0.38, 0.25, 0.32}	                
\definecolor{eggshell}{rgb}{0.94, 0.92, 0.84}	                
\definecolor{egyptianblue}{rgb}{0.06, 0.2, 0.65}	            
\definecolor{electricblue}{rgb}{0.49, 0.98, 1.0}	            
\definecolor{electriccrimson}{rgb}{1.0, 0.0, 0.25}	            
\definecolor{electriccyan}{rgb}{0.0, 1.0, 1.0}	                
\definecolor{electricgreen}{rgb}{0.0, 1.0, 0.0}	                
\definecolor{electricindigo}{rgb}{0.44, 0.0, 1.0}	            
\definecolor{electriclavender}{rgb}{0.96, 0.73, 1.0}	        
\definecolor{electriclime}{rgb}{0.8, 1.0, 0.0}	                
\definecolor{electricpurple}{rgb}{0.75, 0.0, 1.0}	            
\definecolor{electricultramarine}{rgb}{0.25, 0.0, 1.0}	        
\definecolor{electricviolet}{rgb}{0.56, 0.0, 1.0}	            
\definecolor{electricyellow}{rgb}{1.0, 1.0, 0.0}	            
\definecolor{emerald}{rgb}{0.31, 0.78, 0.47}	                
\definecolor{etonblue}{rgb}{0.59, 0.78, 0.64}	                
\definecolor{fallow}{rgb}{0.76, 0.6, 0.42}	                    
\definecolor{falured}{rgb}{0.5, 0.09, 0.09}	                    
\definecolor{fandango}{rgb}{0.71, 0.2, 0.54}	                
\definecolor{fashionfuchsia}{rgb}{0.96, 0.0, 0.63}	            
\definecolor{fawn}{rgb}{0.9, 0.67, 0.44}	                    
\definecolor{feldgrau}{rgb}{0.3, 0.36, 0.33}	                
\definecolor{ferngreen}{rgb}{0.31, 0.47, 0.26}	                
\definecolor{ferrarired}{rgb}{1.0, 0.11, 0.0}	                
\definecolor{fielddrab}{rgb}{0.42, 0.33, 0.12}	                
\definecolor{firebrick}{rgb}{0.7, 0.13, 0.13}	                
\definecolor{fireenginered}{rgb}{0.81, 0.09, 0.13}	            
\definecolor{flame}{rgb}{0.89, 0.35, 0.13}	                    
\definecolor{flamingopink}{rgb}{0.99, 0.56, 0.67}	            
\definecolor{flavescent}{rgb}{0.97, 0.91, 0.56}	                
\definecolor{flax}{rgb}{0.93, 0.86, 0.51}	                    
\definecolor{floralwhite}{rgb}{1.0, 0.98, 0.94}	                
\definecolor{fluorescentorange}{rgb}{1.0, 0.75, 0.0}	        
\definecolor{fluorescentpink}{rgb}{1.0, 0.08, 0.58}	            
\definecolor{fluorescentyellow}{rgb}{0.8, 1.0, 0.0}	            
\definecolor{folly}{rgb}{1.0, 0.0, 0.31}	                    
\definecolor{forestgreen(traditional)}{rgb}{0.0, 0.27, 0.13}	
\definecolor{forestgreen(web)}{rgb}{0.13, 0.55, 0.13}	        
\definecolor{frenchbeige}{rgb}{0.65, 0.48, 0.36}	            
\definecolor{frenchblue}{rgb}{0.0, 0.45, 0.73}	                
\definecolor{frenchlilac}{rgb}{0.53, 0.38, 0.56}	            
\definecolor{frenchrose}{rgb}{0.96, 0.29, 0.54}	                
\definecolor{fuchsia}{rgb}{1.0, 0.0, 1.0}	                    
\definecolor{fuchsiapink}{rgb}{1.0, 0.47, 1.0}	                
\definecolor{fulvous}{rgb}{0.86, 0.52, 0.0}	                    
\definecolor{fuzzywuzzy}{rgb}{0.8, 0.4, 0.4}	                
\definecolor{gainsboro}{rgb}{0.86, 0.86, 0.86}	                
\definecolor{gamboge}{rgb}{0.89, 0.61, 0.06}	                
\definecolor{ghostwhite}{rgb}{0.97, 0.97, 1.0}	                
\definecolor{ginger}{rgb}{0.69, 0.4, 0.0}	                    
\definecolor{glaucous}{rgb}{0.38, 0.51, 0.71}	                
\definecolor{gold(metallic)}{rgb}{0.83, 0.69, 0.22}	            
\definecolor{gold(web)(golden)}{rgb}{1.0, 0.84, 0.0}	        
\definecolor{goldenbrown}{rgb}{0.6, 0.4, 0.08}	                
\definecolor{goldenpoppy}{rgb}{0.99, 0.76, 0.0}	                
\definecolor{goldenyellow}{rgb}{1.0, 0.87, 0.0}	                
\definecolor{goldenrod}{rgb}{0.85, 0.65, 0.13}	                
\definecolor{grannysmithapple}{rgb}{0.66, 0.89, 0.63}	        
\definecolor{gray}{rgb}{0.5, 0.5, 0.5}	                        
\definecolor{gray(html/cssgray)}{rgb}{0.5, 0.5, 0.5}	        
\definecolor{gray(x11gray)}{rgb}{0.75, 0.75, 0.75}	            
\definecolor{gray-asparagus}{rgb}{0.27, 0.35, 0.27}	            
\definecolor{green(colorwheel)(x11green)}{rgb}{0.0, 1.0, 0.0}	
\definecolor{green(html/cssgreen)}{rgb}{0.0, 0.5, 0.0}	        
\definecolor{green(munsell)}{rgb}{0.0, 0.66, 0.47}	            
\definecolor{green(ncs)}{rgb}{0.0, 0.62, 0.42}	                
\definecolor{green(pigment)}{rgb}{0.0, 0.65, 0.31}	            
\definecolor{green(ryb)}{rgb}{0.4, 0.69, 0.2}	                
\definecolor{green-yellow}{rgb}{0.68, 1.0, 0.18}	            
\definecolor{grullo}{rgb}{0.66, 0.6, 0.53}	                    
\definecolor{guppiegreen}{rgb}{0.0, 1.0, 0.5}	                
\definecolor{halayaube}{rgb}{0.4, 0.22, 0.33}	                
\definecolor{hanblue}{rgb}{0.27, 0.42, 0.81}	                
\definecolor{hanpurple}{rgb}{0.32, 0.09, 0.98}	                
\definecolor{hansayellow}{rgb}{0.91, 0.84, 0.42}	            
\definecolor{harlequin}{rgb}{0.25, 1.0, 0.0}	                
\definecolor{harvardcrimson}{rgb}{0.79, 0.0, 0.09}	            
\definecolor{harvestgold}{rgb}{0.85, 0.57, 0.0}	                
\definecolor{heartgold}{rgb}{0.5, 0.5, 0.0}	                    
\definecolor{heliotrope}{rgb}{0.87, 0.45, 1.0}	                
\definecolor{hollywoodcerise}{rgb}{0.96, 0.0, 0.63}	            
\definecolor{honeydew}{rgb}{0.94, 1.0, 0.94}	                
\definecolor{hooker\'sgreen}{rgb}{0.0, 0.44, 0.0}	            
\definecolor{hotmagenta}{rgb}{1.0, 0.11, 0.81}	                
\definecolor{hotpink}{rgb}{1.0, 0.41, 0.71}	                    
\definecolor{huntergreen}{rgb}{0.21, 0.37, 0.23}	            
\definecolor{iceberg}{rgb}{0.44, 0.65, 0.82}	                
\definecolor{icterine}{rgb}{0.99, 0.97, 0.37}	                
\definecolor{inchworm}{rgb}{0.7, 0.93, 0.36}	                
\definecolor{indiagreen}{rgb}{0.07, 0.53, 0.03}	                
\definecolor{indianred}{rgb}{0.8, 0.36, 0.36}	                
\definecolor{indianyellow}{rgb}{0.89, 0.66, 0.34}	            
\definecolor{indigo(dye)}{rgb}{0.0, 0.25, 0.42}	                
\definecolor{indigo(web)}{rgb}{0.29, 0.0, 0.51}	                
\definecolor{internationalkleinblue}{rgb}{0.0, 0.18, 0.65}	    
\definecolor{internationalorange}{rgb}{1.0, 0.31, 0.0}	        
\definecolor{iris}{rgb}{0.35, 0.31, 0.81}	                    
\definecolor{isabelline}{rgb}{0.96, 0.94, 0.93}	                
\definecolor{islamicgreen}{rgb}{0.0, 0.56, 0.0}	                
\definecolor{ivory}{rgb}{1.0, 1.0, 0.94}	                    
\definecolor{jade}{rgb}{0.0, 0.66, 0.42}	                    
\definecolor{jasper}{rgb}{0.84, 0.23, 0.24}	                    
\definecolor{jazzberryjam}{rgb}{0.65, 0.04, 0.37}	            
\definecolor{jonquil}{rgb}{0.98, 0.85, 0.37}	                
\definecolor{junebud}{rgb}{0.74, 0.85, 0.34}	                
\definecolor{junglegreen}{rgb}{0.16, 0.67, 0.53}	            
\definecolor{kellygreen}{rgb}{0.3, 0.73, 0.09}	                
\definecolor{khaki(html/css)(khaki)}{rgb}{0.76, 0.69, 0.57}	    
\definecolor{khaki(x11)(lightkhaki)}{rgb}{0.94, 0.9, 0.55}	    
\definecolor{lasallegreen}{rgb}{0.03, 0.47, 0.19}	            
\definecolor{languidlavender}{rgb}{0.84, 0.79, 0.87}	        
\definecolor{lapislazuli}{rgb}{0.15, 0.38, 0.61}	            
\definecolor{laserlemon}{rgb}{1.0, 1.0, 0.13}	                
\definecolor{lava}{rgb}{0.81, 0.06, 0.13}	                    
\definecolor{lavender(floral)}{rgb}{0.71, 0.49, 0.86}	        
\definecolor{lavender(web)}{rgb}{0.9, 0.9, 0.98}	            
\definecolor{lavenderblue}{rgb}{0.8, 0.8, 1.0}	                
\definecolor{lavenderblush}{rgb}{1.0, 0.94, 0.96}	            
\definecolor{lavendergray}{rgb}{0.77, 0.76, 0.82}	            
\definecolor{lavenderindigo}{rgb}{0.58, 0.34, 0.92}	            
\definecolor{lavendermagenta}{rgb}{0.93, 0.51, 0.93}	        
\definecolor{lavendermist}{rgb}{0.9, 0.9, 0.98}	                
\definecolor{lavenderpink}{rgb}{0.98, 0.68, 0.82}	            
\definecolor{lavenderpurple}{rgb}{0.59, 0.48, 0.71}	            
\definecolor{lavenderrose}{rgb}{0.98, 0.63, 0.89}	            
\definecolor{lawngreen}{rgb}{0.49, 0.99, 0.0}	                
\definecolor{lemon}{rgb}{1.0, 0.97, 0.0}	                    
\definecolor{lemonchiffon}{rgb}{1.0, 0.98, 0.8}	                
\definecolor{lightapricot}{rgb}{0.99, 0.84, 0.69}	            
\definecolor{lightblue}{rgb}{0.68, 0.85, 0.9}	                
\definecolor{lightbrown}{rgb}{0.71, 0.4, 0.11}	                
\definecolor{lightcarminepink}{rgb}{0.9, 0.4, 0.38}	            
\definecolor{lightcoral}{rgb}{0.94, 0.5, 0.5}	                
\definecolor{lightcornflowerblue}{rgb}{0.6, 0.81, 0.93}	        
\definecolor{lightcyan}{rgb}{0.88, 1.0, 1.0}	                
\definecolor{lightfuchsiapink}{rgb}{0.98, 0.52, 0.9}	        
\definecolor{lightgoldenrodyellow}{rgb}{0.98, 0.98, 0.82}	    
\definecolor{lightgray}{rgb}{0.83, 0.83, 0.83}	                
\definecolor{lightgreen}{rgb}{0.56, 0.93, 0.56}	                
\definecolor{lightkhaki}{rgb}{0.94, 0.9, 0.55}	                
\definecolor{lightmauve}{rgb}{0.86, 0.82, 1.0}	                
\definecolor{lightpastelpurple}{rgb}{0.69, 0.61, 0.85}	        
\definecolor{lightpink}{rgb}{1.0, 0.71, 0.76}	                
\definecolor{lightsalmon}{rgb}{1.0, 0.63, 0.48}	                
\definecolor{lightsalmonpink}{rgb}{1.0, 0.6, 0.6}	            
\definecolor{lightseagreen}{rgb}{0.13, 0.7, 0.67}	            
\definecolor{lightskyblue}{rgb}{0.53, 0.81, 0.98}	            
\definecolor{lightslategray}{rgb}{0.47, 0.53, 0.6}	            
\definecolor{lighttaupe}{rgb}{0.7, 0.55, 0.43}	                
\definecolor{lightthulianpink}{rgb}{0.9, 0.56, 0.67}	        
\definecolor{lightyellow}{rgb}{1.0, 1.0, 0.88}	                
\definecolor{lilac}{rgb}{0.78, 0.64, 0.78}	                    
\definecolor{lime(colorwheel)}{rgb}{0.75, 1.0, 0.0}	            
\definecolor{lime(web)(x11green)}{rgb}{0.0, 1.0, 0.0}	        
\definecolor{limegreen}{rgb}{0.2, 0.8, 0.2}	                    
\definecolor{lincolngreen}{rgb}{0.11, 0.35, 0.02}	            
\definecolor{linen}{rgb}{0.98, 0.94, 0.9}	                    
\definecolor{liver}{rgb}{0.33, 0.29, 0.31}	                    
\definecolor{lust}{rgb}{0.9, 0.13, 0.13}	                    
\definecolor{macaroniandcheese}{rgb}{1.0, 0.74, 0.53}	        
\definecolor{magenta}{rgb}{1.0, 0.0, 1.0}	                    
\definecolor{magenta(dye)}{rgb}{0.79, 0.08, 0.48}	            
\definecolor{magenta(process)}{rgb}{1.0, 0.0, 0.56}	            
\definecolor{magicmint}{rgb}{0.67, 0.94, 0.82}	                
\definecolor{magnolia}{rgb}{0.97, 0.96, 1.0}	                
\definecolor{mahogany}{rgb}{0.75, 0.25, 0.0}	                
\definecolor{maize}{rgb}{0.98, 0.93, 0.37}	                    
\definecolor{majorelleblue}{rgb}{0.38, 0.31, 0.86}	            
\definecolor{malachite}{rgb}{0.04, 0.85, 0.32}	                
\definecolor{manatee}{rgb}{0.59, 0.6, 0.67}	                    
\definecolor{mangotango}{rgb}{1.0, 0.51, 0.26}	                
\definecolor{maroon(html/css)}{rgb}{0.5, 0.0, 0.0}	            
\definecolor{maroon(x11)}{rgb}{0.69, 0.19, 0.38}	            
\definecolor{mauve}{rgb}{0.88, 0.69, 1.0}	                    
\definecolor{mauvetaupe}{rgb}{0.57, 0.37, 0.43}	                
\definecolor{mauvelous}{rgb}{0.94, 0.6, 0.67}	                
\definecolor{mayablue}{rgb}{0.45, 0.76, 0.98}	                
\definecolor{meatbrown}{rgb}{0.9, 0.72, 0.23}	                
\definecolor{mediumaquamarine}{rgb}{0.4, 0.8, 0.67}	            
\definecolor{mediumblue}{rgb}{0.0, 0.0, 0.8}	                
\definecolor{mediumcandyapplered}{rgb}{0.89, 0.02, 0.17}	    
\definecolor{mediumcarmine}{rgb}{0.69, 0.25, 0.21}	            
\definecolor{mediumchampagne}{rgb}{0.95, 0.9, 0.67}	            
\definecolor{mediumelectricblue}{rgb}{0.01, 0.31, 0.59}	        
\definecolor{mediumjunglegreen}{rgb}{0.11, 0.21, 0.18}	        
\definecolor{mediumlavendermagenta}{rgb}{0.8, 0.6, 0.8}	        
\definecolor{mediumorchid}{rgb}{0.73, 0.33, 0.83}	            
\definecolor{mediumpersianblue}{rgb}{0.0, 0.4, 0.65}	        
\definecolor{mediumpurple}{rgb}{0.58, 0.44, 0.86}	            
\definecolor{mediumred-violet}{rgb}{0.73, 0.2, 0.52}	        
\definecolor{mediumseagreen}{rgb}{0.24, 0.7, 0.44}	            
\definecolor{mediumslateblue}{rgb}{0.48, 0.41, 0.93}	        
\definecolor{mediumspringbud}{rgb}{0.79, 0.86, 0.54}	        
\definecolor{mediumspringgreen}{rgb}{0.0, 0.98, 0.6}	        
\definecolor{mediumtaupe}{rgb}{0.4, 0.3, 0.28}	                
\definecolor{mediumtealblue}{rgb}{0.0, 0.33, 0.71}	            
\definecolor{mediumturquoise}{rgb}{0.28, 0.82, 0.8}	            
\definecolor{mediumviolet-red}{rgb}{0.78, 0.08, 0.52}	        
\definecolor{melon}{rgb}{0.99, 0.74, 0.71}	                    
\definecolor{midnightblue}{rgb}{0.1, 0.1, 0.44}	                
\definecolor{midnightgreen(eaglegreen)}{rgb}{0.0, 0.29, 0.33}	
\definecolor{mikadoyellow}{rgb}{1.0, 0.77, 0.05}	            
\definecolor{mint}{rgb}{0.24, 0.71, 0.54}	                    
\definecolor{mintcream}{rgb}{0.96, 1.0, 0.98}	                
\definecolor{mintgreen}{rgb}{0.6, 1.0, 0.6}	                    
\definecolor{mistyrose}{rgb}{1.0, 0.89, 0.88}	                
\definecolor{moccasin}{rgb}{0.98, 0.92, 0.84}	                
\definecolor{modebeige}{rgb}{0.59, 0.44, 0.09}	                
\definecolor{moonstoneblue}{rgb}{0.45, 0.66, 0.76}	            
\definecolor{mordantred19}{rgb}{0.68, 0.05, 0.0}	            
\definecolor{mossgreen}{rgb}{0.68, 0.87, 0.68}	                
\definecolor{mountainmeadow}{rgb}{0.19, 0.73, 0.56}	            
\definecolor{mountbattenpink}{rgb}{0.6, 0.48, 0.55}	            
\definecolor{mulberry}{rgb}{0.77, 0.29, 0.55}	                
\definecolor{mustard}{rgb}{1.0, 0.86, 0.35}	                    
\definecolor{myrtle}{rgb}{0.13, 0.26, 0.12}	                    
\definecolor{msugreen}{rgb}{0.09, 0.27, 0.23}	                
\definecolor{nadeshikopink}{rgb}{0.96, 0.68, 0.78}	            
\definecolor{napiergreen}{rgb}{0.16, 0.5, 0.0}	                
\definecolor{naplesyellow}{rgb}{0.98, 0.85, 0.37}	            
\definecolor{navajowhite}{rgb}{1.0, 0.87, 0.68}	                
\definecolor{navyblue}{rgb}{0.0, 0.0, 0.5}	                    
\definecolor{neoncarrot}{rgb}{1.0, 0.64, 0.26}	                
\definecolor{neonfuchsia}{rgb}{1.0, 0.25, 0.39}	                
\definecolor{neongreen}{rgb}{0.22, 0.88, 0.08}	                
\definecolor{non-photoblue}{rgb}{0.64, 0.87, 0.93}	            
\definecolor{oceanboatblue}{rgb}{0.0, 0.47, 0.75}	            
\definecolor{ochre}{rgb}{0.8, 0.47, 0.13}	                    
\definecolor{officegreen}{rgb}{0.0, 0.5, 0.0}	                
\definecolor{oldgold}{rgb}{0.81, 0.71, 0.23}	                
\definecolor{oldlace}{rgb}{0.99, 0.96, 0.9}	                    
\definecolor{oldlavender}{rgb}{0.47, 0.41, 0.47}	            
\definecolor{oldmauve}{rgb}{0.4, 0.19, 0.28}	                
\definecolor{oldrose}{rgb}{0.75, 0.5, 0.51}	                    
\definecolor{olive}{rgb}{0.5, 0.5, 0.0}	                        
\definecolor{olivedrab(web)(olivedrab3)}{rgb}{0.42, 0.56, 0.14}	
\definecolor{olivedrab7}{rgb}{0.24, 0.2, 0.12}	                
\definecolor{olivine}{rgb}{0.6, 0.73, 0.45}	                    
\definecolor{onyx}{rgb}{0.06, 0.06, 0.06}	                    
\definecolor{operamauve}{rgb}{0.72, 0.52, 0.65}	                
\definecolor{orange(colorwheel)}{rgb}{1.0, 0.5, 0.0}	        
\definecolor{orange(ryb)}{rgb}{0.98, 0.6, 0.01}	                
\definecolor{orange(webcolor)}{rgb}{1.0, 0.65, 0.0}	            
\definecolor{orangepeel}{rgb}{1.0, 0.62, 0.0}	                
\definecolor{orange-red}{rgb}{1.0, 0.27, 0.0}	                
\definecolor{orchid}{rgb}{0.85, 0.44, 0.84}	                    
\definecolor{otterbrown}{rgb}{0.4, 0.26, 0.13}	                
\definecolor{outerspace}{rgb}{0.25, 0.29, 0.3}	                
\definecolor{outrageousorange}{rgb}{1.0, 0.43, 0.29}	        
\definecolor{oxfordblue}{rgb}{0.0, 0.13, 0.28}	                
\definecolor{oucrimsonred}{rgb}{0.6, 0.0, 0.0}	                
\definecolor{pakistangreen}{rgb}{0.0, 0.4, 0.0}	                
\definecolor{palatinateblue}{rgb}{0.15, 0.23, 0.89}	            
\definecolor{palatinatepurple}{rgb}{0.41, 0.16, 0.38}	        
\definecolor{paleaqua}{rgb}{0.74, 0.83, 0.9}	                
\definecolor{paleblue}{rgb}{0.69, 0.93, 0.93}	                
\definecolor{palebrown}{rgb}{0.6, 0.46, 0.33}	                
\definecolor{palecarmine}{rgb}{0.69, 0.25, 0.21}	            
\definecolor{palecerulean}{rgb}{0.61, 0.77, 0.89}	            
\definecolor{palechestnut}{rgb}{0.87, 0.68, 0.69}	            
\definecolor{palecopper}{rgb}{0.85, 0.54, 0.4}	                
\definecolor{palecornflowerblue}{rgb}{0.67, 0.8, 0.94}	        
\definecolor{palegold}{rgb}{0.9, 0.75, 0.54}	                
\definecolor{palegoldenrod}{rgb}{0.93, 0.91, 0.67}	            
\definecolor{palegreen}{rgb}{0.6, 0.98, 0.6}	                
\definecolor{palemagenta}{rgb}{0.98, 0.52, 0.9}	                
\definecolor{palepink}{rgb}{0.98, 0.85, 0.87}	                
\definecolor{paleplum}{rgb}{0.8, 0.6, 0.8}	                    
\definecolor{palered-violet}{rgb}{0.86, 0.44, 0.58}	            
\definecolor{palerobineggblue}{rgb}{0.59, 0.87, 0.82}	        
\definecolor{palesilver}{rgb}{0.79, 0.75, 0.73}	                
\definecolor{palespringbud}{rgb}{0.93, 0.92, 0.74}	            
\definecolor{paletaupe}{rgb}{0.74, 0.6, 0.49}	                
\definecolor{paleviolet-red}{rgb}{0.86, 0.44, 0.58}	            
\definecolor{pansypurple}{rgb}{0.47, 0.09, 0.29}	            
\definecolor{papayawhip}{rgb}{1.0, 0.94, 0.84}	                
\definecolor{parisgreen}{rgb}{0.31, 0.78, 0.47}	                
\definecolor{pastelblue}{rgb}{0.68, 0.78, 0.81}	                
\definecolor{pastelbrown}{rgb}{0.51, 0.41, 0.33}	            
\definecolor{pastelgray}{rgb}{0.81, 0.81, 0.77}	                
\definecolor{pastelgreen}{rgb}{0.47, 0.87, 0.47}	            
\definecolor{pastelmagenta}{rgb}{0.96, 0.6, 0.76}	            
\definecolor{pastelorange}{rgb}{1.0, 0.7, 0.28}	                
\definecolor{pastelpink}{rgb}{1.0, 0.82, 0.86}	                
\definecolor{pastelpurple}{rgb}{0.7, 0.62, 0.71}	            
\definecolor{pastelred}{rgb}{1.0, 0.41, 0.38}	                
\definecolor{pastelviolet}{rgb}{0.8, 0.6, 0.79}	                
\definecolor{pastelyellow}{rgb}{0.99, 0.99, 0.59}	            
\definecolor{patriarch}{rgb}{0.5, 0.0, 0.5}	                    
\definecolor{payne\'sgrey}{rgb}{0.25, 0.25, 0.28}	            
\definecolor{peach}{rgb}{1.0, 0.9, 0.71}	                    
\definecolor{peach-orange}{rgb}{1.0, 0.8, 0.6}	                
\definecolor{peachpuff}{rgb}{1.0, 0.85, 0.73}	                
\definecolor{peach-yellow}{rgb}{0.98, 0.87, 0.68}	            
\definecolor{pear}{rgb}{0.82, 0.89, 0.19}	                    
\definecolor{pearl}{rgb}{0.94, 0.92, 0.84}	                    
\definecolor{peridot}{rgb}{0.9, 0.89, 0.0}	                    
\definecolor{periwinkle}{rgb}{0.8, 0.8, 1.0}	                
\definecolor{persianblue}{rgb}{0.11, 0.22, 0.73}	            
\definecolor{persiangreen}{rgb}{0.0, 0.65, 0.58}	            
\definecolor{persianindigo}{rgb}{0.2, 0.07, 0.48}	            
\definecolor{persianorange}{rgb}{0.85, 0.56, 0.35}	            
\definecolor{peru}{rgb}{0.8, 0.52, 0.25}	                    
\definecolor{persianpink}{rgb}{0.97, 0.5, 0.75}	                
\definecolor{persianplum}{rgb}{0.44, 0.11, 0.11}	            
\definecolor{persianred}{rgb}{0.8, 0.2, 0.2}	                
\definecolor{persianrose}{rgb}{1.0, 0.16, 0.64}	                
\definecolor{persimmon}{rgb}{0.93, 0.35, 0.0}	                
\definecolor{phlox}{rgb}{0.87, 0.0, 1.0}	                    
\definecolor{phthaloblue}{rgb}{0.0, 0.06, 0.54}	                
\definecolor{phthalogreen}{rgb}{0.07, 0.21, 0.14}	            
\definecolor{piggypink}{rgb}{0.99, 0.87, 0.9}	                
\definecolor{pinegreen}{rgb}{0.0, 0.47, 0.44}	                
\definecolor{pink}{rgb}{1.0, 0.75, 0.8}	                        
\definecolor{pink-orange}{rgb}{1.0, 0.6, 0.4}	                
\definecolor{pinkpearl}{rgb}{0.91, 0.67, 0.81}	                
\definecolor{pinksherbet}{rgb}{0.97, 0.56, 0.65}	            
\definecolor{pistachio}{rgb}{0.58, 0.77, 0.45}	                
\definecolor{platinum}{rgb}{0.9, 0.89, 0.89}	                
\definecolor{plum(traditional)}{rgb}{0.56, 0.27, 0.52}	        
\definecolor{plum(web)}{rgb}{0.8, 0.6, 0.8}	                    
\definecolor{portlandorange}{rgb}{1.0, 0.35, 0.21}	            
\definecolor{powderblue(web)}{rgb}{0.69, 0.88, 0.9}	            
\definecolor{princetonorange}{rgb}{1.0, 0.56, 0.0}	            
\definecolor{prune}{rgb}{0.44, 0.11, 0.11}	                    
\definecolor{prussianblue}{rgb}{0.0, 0.19, 0.33}	            
\definecolor{psychedelicpurple}{rgb}{0.87, 0.0, 1.0}	        
\definecolor{puce}{rgb}{0.8, 0.53, 0.6}	                        
\definecolor{pumpkin}{rgb}{1.0, 0.46, 0.09}	                    
\definecolor{purple(html/css)}{rgb}{0.5, 0.0, 0.5}	            
\definecolor{purple(munsell)}{rgb}{0.62, 0.0, 0.77}	            
\definecolor{purple(x11)}{rgb}{0.63, 0.36, 0.94}	            
\definecolor{purpleheart}{rgb}{0.41, 0.21, 0.61}	            
\definecolor{purplemountainmajesty}{rgb}{0.59, 0.47, 0.71}	    
\definecolor{purplepizzazz}{rgb}{1.0, 0.31, 0.85}	            
\definecolor{purpletaupe}{rgb}{0.31, 0.25, 0.3}	                
\definecolor{radicalred}{rgb}{1.0, 0.21, 0.37}	                
\definecolor{raspberry}{rgb}{0.89, 0.04, 0.36}	                
\definecolor{raspberryglace}{rgb}{0.57, 0.37, 0.43}	            
\definecolor{raspberrypink}{rgb}{0.89, 0.31, 0.61}	            
\definecolor{raspberryrose}{rgb}{0.7, 0.27, 0.42}	            
\definecolor{rawumber}{rgb}{0.51, 0.4, 0.27}	                
\definecolor{razzledazzlerose}{rgb}{1.0, 0.2, 0.8}	            
\definecolor{razzmatazz}{rgb}{0.89, 0.15, 0.42}	                
\definecolor{red}{rgb}{1.0, 0.0, 0.0}	                        
\definecolor{red(munsell)}{rgb}{0.95, 0.0, 0.24}	            
\definecolor{red(ncs)}{rgb}{0.77, 0.01, 0.2}	                
\definecolor{red(pigment)}{rgb}{0.93, 0.11, 0.14}	            
\definecolor{red(ryb)}{rgb}{1.0, 0.15, 0.07}	                
\definecolor{red-brown}{rgb}{0.65, 0.16, 0.16}	                
\definecolor{red-violet}{rgb}{0.78, 0.08, 0.52}	                
\definecolor{redwood}{rgb}{0.67, 0.31, 0.32}	                
\definecolor{regalia}{rgb}{0.32, 0.18, 0.5}	                    
\definecolor{richblack}{rgb}{0.0, 0.25, 0.25}	                
\definecolor{richbrilliantlavender}{rgb}{0.95, 0.65, 1.0}	    
\definecolor{richcarmine}{rgb}{0.84, 0.0, 0.25}	                
\definecolor{richelectricblue}{rgb}{0.03, 0.57, 0.82}	        
\definecolor{richlavender}{rgb}{0.67, 0.38, 0.8}	            
\definecolor{richlilac}{rgb}{0.71, 0.4, 0.82}	                
\definecolor{richmaroon}{rgb}{0.69, 0.19, 0.38}	                
\definecolor{riflegreen}{rgb}{0.25, 0.28, 0.2}	                
\definecolor{robineggblue}{rgb}{0.0, 0.8, 0.8}	                
\definecolor{rose}{rgb}{1.0, 0.0, 0.5}	                        
\definecolor{rosebonbon}{rgb}{0.98, 0.26, 0.62}	                
\definecolor{roseebony}{rgb}{0.4, 0.3, 0.28}	                
\definecolor{rosegold}{rgb}{0.72, 0.43, 0.47}	                
\definecolor{rosemadder}{rgb}{0.89, 0.15, 0.21}	                
\definecolor{rosepink}{rgb}{1.0, 0.4, 0.8}	                    
\definecolor{rosequartz}{rgb}{0.67, 0.6, 0.66}	                
\definecolor{rosetaupe}{rgb}{0.56, 0.36, 0.36}	                
\definecolor{rosevale}{rgb}{0.67, 0.31, 0.32}	                
\definecolor{rosewood}{rgb}{0.4, 0.0, 0.04}	                    
\definecolor{rossocorsa}{rgb}{0.83, 0.0, 0.0}	                
\definecolor{rosybrown}{rgb}{0.74, 0.56, 0.56}	                
\definecolor{royalazure}{rgb}{0.0, 0.22, 0.66}	                
\definecolor{royalblue(traditional)}{rgb}{0.0, 0.14, 0.4}	    
\definecolor{royalblue(web)}{rgb}{0.25, 0.41, 0.88}	            
\definecolor{royalfuchsia}{rgb}{0.79, 0.17, 0.57}	            
\definecolor{royalpurple}{rgb}{0.47, 0.32, 0.66}	            
\definecolor{ruby}{rgb}{0.88, 0.07, 0.37}	                    
\definecolor{ruddy}{rgb}{1.0, 0.0, 0.16}	                    
\definecolor{ruddybrown}{rgb}{0.73, 0.4, 0.16}	                
\definecolor{ruddypink}{rgb}{0.88, 0.56, 0.59}	                
\definecolor{rufous}{rgb}{0.66, 0.11, 0.03}	                    
\definecolor{russet}{rgb}{0.5, 0.27, 0.11}	                    
\definecolor{rust}{rgb}{0.72, 0.25, 0.05}	                    
\definecolor{sacramentostategreen}{rgb}{0.0, 0.34, 0.25}	    
\definecolor{saddlebrown}{rgb}{0.55, 0.27, 0.07}	            
\definecolor{safetyorange(blazeorange)}{rgb}{1.0, 0.4, 0.0}	    
\definecolor{saffron}{rgb}{0.96, 0.77, 0.19}	                
\definecolor{st.patrick\'sblue}{rgb}{0.14, 0.16, 0.48}	        
\definecolor{salmon}{rgb}{1.0, 0.55, 0.41}	                    
\definecolor{salmonpink}{rgb}{1.0, 0.57, 0.64}	                
\definecolor{sand}{rgb}{0.76, 0.7, 0.5}	                        
\definecolor{sanddune}{rgb}{0.59, 0.44, 0.09}	                
\definecolor{sandstorm}{rgb}{0.93, 0.84, 0.25}	                
\definecolor{sandybrown}{rgb}{0.96, 0.64, 0.38}	                
\definecolor{sandytaupe}{rgb}{0.59, 0.44, 0.09}	                
\definecolor{sangria}{rgb}{0.57, 0.0, 0.04}	                    
\definecolor{sapgreen}{rgb}{0.31, 0.49, 0.16}	                
\definecolor{sapphire}{rgb}{0.03, 0.15, 0.4}	                
\definecolor{satinsheengold}{rgb}{0.8, 0.63, 0.21}	            
\definecolor{scarlet}{rgb}{1.0, 0.13, 0.0}	                    
\definecolor{schoolbusyellow}{rgb}{1.0, 0.85, 0.0}	            
\definecolor{screamin\'green}{rgb}{0.46, 1.0, 0.44}	            
\definecolor{seagreen}{rgb}{0.18, 0.55, 0.34}	                
\definecolor{sealbrown}{rgb}{0.2, 0.08, 0.08}	                
\definecolor{seashell}{rgb}{1.0, 0.96, 0.93}	                
\definecolor{selectiveyellow}{rgb}{1.0, 0.73, 0.0}	            
\definecolor{sepia}{rgb}{0.44, 0.26, 0.08}	                    
\definecolor{shadow}{rgb}{0.54, 0.47, 0.36}	                    
\definecolor{shamrockgreen}{rgb}{0.0, 0.62, 0.38}	            
\definecolor{shockingpink}{rgb}{0.99, 0.06, 0.75}	            
\definecolor{sienna}{rgb}{0.53, 0.18, 0.09}	                    
\definecolor{silver}{rgb}{0.75, 0.75, 0.75}	                    
\definecolor{sinopia}{rgb}{0.8, 0.25, 0.04}	                    
\definecolor{skobeloff}{rgb}{0.0, 0.48, 0.45}	                
\definecolor{skyblue}{rgb}{0.53, 0.81, 0.92}	                
\definecolor{skymagenta}{rgb}{0.81, 0.44, 0.69}	                
\definecolor{slateblue}{rgb}{0.42, 0.35, 0.8}	                
\definecolor{slategray}{rgb}{0.44, 0.5, 0.56}	                
\definecolor{smalt(darkpowderblue)}{rgb}{0.0, 0.2, 0.6}	        
\definecolor{smokeytopaz}{rgb}{0.58, 0.25, 0.03}	            
\definecolor{smokyblack}{rgb}{0.06, 0.05, 0.03}	                
\definecolor{snow}{rgb}{1.0, 0.98, 0.98}	                    
\definecolor{spirodiscoball}{rgb}{0.06, 0.75, 0.99}	            
\definecolor{splashedwhite}{rgb}{1.0, 0.99, 1.0}	            
\definecolor{springbud}{rgb}{0.65, 0.99, 0.0}	                
\definecolor{springgreen}{rgb}{0.0, 1.0, 0.5}	                
\definecolor{steelblue}{rgb}{0.27, 0.51, 0.71}	                
\definecolor{stildegrainyellow}{rgb}{0.98, 0.85, 0.37}	        
\definecolor{straw}{rgb}{0.89, 0.85, 0.44}	                    
\definecolor{sunglow}{rgb}{1.0, 0.8, 0.2}	                    
\definecolor{sunset}{rgb}{0.98, 0.84, 0.65}	                    
\definecolor{tan}{rgb}{0.82, 0.71, 0.55}	                    
\definecolor{tangelo}{rgb}{0.98, 0.3, 0.0}	                    
\definecolor{tangerine}{rgb}{0.95, 0.52, 0.0}	                
\definecolor{tangerineyellow}{rgb}{1.0, 0.8, 0.0}	            
\definecolor{taupe}{rgb}{0.28, 0.24, 0.2}	                    
\definecolor{taupegray}{rgb}{0.55, 0.52, 0.54}	                
\definecolor{teagreen}{rgb}{0.82, 0.94, 0.75}	                
\definecolor{tearose(orange)}{rgb}{0.97, 0.51, 0.47}	        
\definecolor{tearose(rose)}{rgb}{0.96, 0.76, 0.76}	            
\definecolor{teal}{rgb}{0.0, 0.5, 0.5}	                        
\definecolor{tealblue}{rgb}{0.21, 0.46, 0.53}	                
\definecolor{tealgreen}{rgb}{0.0, 0.51, 0.5}	                
\definecolor{tenne-tawny}{rgb}{0.8, 0.34, 0.0}	                
\definecolor{terracotta}{rgb}{0.89, 0.45, 0.36}	                
\definecolor{thistle}{rgb}{0.85, 0.75, 0.85}	                
\definecolor{thulianpink}{rgb}{0.87, 0.44, 0.63}	            
\definecolor{ticklemepink}{rgb}{0.99, 0.54, 0.67}	            
\definecolor{tiffanyblue}{rgb}{0.04, 0.73, 0.71}	            
\definecolor{tiger\'seye}{rgb}{0.88, 0.55, 0.24}	            
\definecolor{timberwolf}{rgb}{0.86, 0.84, 0.82}	                
\definecolor{titaniumyellow}{rgb}{0.93, 0.9, 0.0}	            
\definecolor{tomato}{rgb}{1.0, 0.39, 0.28}	                    
\definecolor{toolbox}{rgb}{0.45, 0.42, 0.75}	                
\definecolor{tractorred}{rgb}{0.99, 0.05, 0.21}	                
\definecolor{trolleygrey}{rgb}{0.5, 0.5, 0.5}	                
\definecolor{tropicalrainforest}{rgb}{0.0, 0.46, 0.37}	        
\definecolor{trueblue}{rgb}{0.0, 0.45, 0.81}	                
\definecolor{tuftsblue}{rgb}{0.28, 0.57, 0.81}	                
\definecolor{tumbleweed}{rgb}{0.87, 0.67, 0.53}	                
\definecolor{turkishrose}{rgb}{0.71, 0.45, 0.51}	            
\definecolor{turquoise}{rgb}{0.19, 0.84, 0.78}	                
\definecolor{turquoiseblue}{rgb}{0.0, 1.0, 0.94}	            
\definecolor{turquoisegreen}{rgb}{0.63, 0.84, 0.71}	            
\definecolor{tuscanred}{rgb}{0.51, 0.21, 0.21}	                
\definecolor{twilightlavender}{rgb}{0.54, 0.29, 0.42}	        
\definecolor{tyrianpurple}{rgb}{0.4, 0.01, 0.24}	            
\definecolor{uablue}{rgb}{0.0, 0.2, 0.67}	                    
\definecolor{uared}{rgb}{0.85, 0.0, 0.3}	                    
\definecolor{ube}{rgb}{0.53, 0.47, 0.76}	                    
\definecolor{uclablue}{rgb}{0.33, 0.41, 0.58}	                
\definecolor{uclagold}{rgb}{1.0, 0.7, 0.0}	                    
\definecolor{ufogreen}{rgb}{0.24, 0.82, 0.44}	                
\definecolor{ultramarine}{rgb}{0.07, 0.04, 0.56}	            
\definecolor{ultramarineblue}{rgb}{0.25, 0.4, 0.96}	            
\definecolor{ultrapink}{rgb}{1.0, 0.44, 1.0}	                
\definecolor{umber}{rgb}{0.39, 0.32, 0.28}	                    
\definecolor{unitednationsblue}{rgb}{0.36, 0.57, 0.9}	        
\definecolor{unmellowyellow}{rgb}{1.0, 1.0, 0.4}	            
\definecolor{upforestgreen}{rgb}{0.0, 0.27, 0.13}	            
\definecolor{upmaroon}{rgb}{0.48, 0.07, 0.07}	                
\definecolor{upsdellred}{rgb}{0.68, 0.09, 0.13}	                
\definecolor{urobilin}{rgb}{0.88, 0.68, 0.13}	                
\definecolor{usccardinal}{rgb}{0.6, 0.0, 0.0}	                
\definecolor{uscgold}{rgb}{1.0, 0.8, 0.0}	                    
\definecolor{utahcrimson}{rgb}{0.83, 0.0, 0.25}	                
\definecolor{vanilla}{rgb}{0.95, 0.9, 0.67}	                    
\definecolor{vegasgold}{rgb}{0.77, 0.7, 0.35}	                
\definecolor{venetianred}{rgb}{0.78, 0.03, 0.08}	            
\definecolor{verdigris}{rgb}{0.26, 0.7, 0.68}	                
\definecolor{vermilion}{rgb}{0.89, 0.26, 0.2}	                
\definecolor{veronica}{rgb}{0.63, 0.36, 0.94}	                
\definecolor{violet}{rgb}{0.56, 0.0, 1.0}	                    
\definecolor{violet(colorwheel)}{rgb}{0.5, 0.0, 1.0}	        
\definecolor{violet(ryb)}{rgb}{0.53, 0.0, 0.69}	                
\definecolor{violet(web)}{rgb}{0.93, 0.51, 0.93}	            
\definecolor{viridian}{rgb}{0.25, 0.51, 0.43}	                
\definecolor{vividauburn}{rgb}{0.58, 0.15, 0.14}	            
\definecolor{vividburgundy}{rgb}{0.62, 0.11, 0.21}	            
\definecolor{vividcerise}{rgb}{0.85, 0.11, 0.51}	            
\definecolor{vividtangerine}{rgb}{1.0, 0.63, 0.54}	            
\definecolor{vividviolet}{rgb}{0.62, 0.0, 1.0}	                
\definecolor{warmblack}{rgb}{0.0, 0.26, 0.26}	                
\definecolor{wenge}{rgb}{0.39, 0.33, 0.32}	                    
\definecolor{wheat}{rgb}{0.96, 0.87, 0.7}	                    
\definecolor{white}{rgb}{1.0, 1.0, 1.0}	                        
\definecolor{whitesmoke}{rgb}{0.96, 0.96, 0.96}	                
\definecolor{wildblueyonder}{rgb}{0.64, 0.68, 0.82}	            
\definecolor{wildstrawberry}{rgb}{1.0, 0.26, 0.64}	            
\definecolor{wildwatermelon}{rgb}{0.99, 0.42, 0.52}	            
\definecolor{wisteria}{rgb}{0.79, 0.63, 0.86}	                
\definecolor{xanadu}{rgb}{0.45, 0.53, 0.47}	                    
\definecolor{yaleblue}{rgb}{0.06, 0.3, 0.57}	                
\definecolor{yellow}{rgb}{1.0, 1.0, 0.0}	                    
\definecolor{yellow(munsell)}{rgb}{0.94, 0.8, 0.0}	            
\definecolor{yellow(ncs)}{rgb}{1.0, 0.83, 0.0}	                
\definecolor{yellow(process)}{rgb}{1.0, 0.94, 0.0}	            
\definecolor{yellow(ryb)}{rgb}{1.0, 1.0, 0.2}	                
\definecolor{yellow-green}{rgb}{0.6, 0.8, 0.2}	                
\definecolor{zaffre}{rgb}{0.0, 0.08, 0.66}	                    
\definecolor{zinnwalditebrown}{rgb}{0.17, 0.09, 0.03}	        

\definecolor{solarized-yellow}{HTML}	{B58900}
\definecolor{solarized-orange}{HTML}	{CB4B16}
\definecolor{solarized-red}{HTML}		{DC322F}
\definecolor{solarized-magenta}{HTML}	{D33682}
\definecolor{solarized-violet}{HTML}	{6C71C4}
\definecolor{solarized-blue}{HTML}		{268BD2}
\definecolor{solarized-cyan}{HTML}		{2AA198}
\definecolor{solarized-green}{HTML}		{859900}

\definecolor{solarized-base03}{HTML}	{002B36}
\definecolor{solarized-base02}{HTML}	{073642}
\definecolor{solarized-base01}{HTML}	{586E75}
\definecolor{solarized-base00}{HTML}	{657B83}
\definecolor{solarized-base0}{HTML}		{839496}
\definecolor{solarized-base1}{HTML}		{93A1A1}
\definecolor{solarized-base2}{HTML}		{EEE8D5}
\definecolor{solarized-base3}{HTML}		{FDF6E3}

\usepackage[
            pdfauthor={Andrey Rukhin, PhD},
            pdftitle={On the Parallel Tower of Hanoi Puzzle: Acyclicity and a Conditional Triangle Inequality},
            pdfsubject={The Parallel Tower of Hanoi Problem},
            pdfkeywords={Tower of Hanoi, Multi-peg Tower of Hanoi, Parallel Tower of Hanoi},
            pdfcreator={pdflatex}]{hyperref}
\usepackage[skip=6pt]{caption}
\usepackage{tikz}
\usetikzlibrary{arrows,decorations.pathmorphing,backgrounds,positioning,fit,petri}
\usepackage{enumerate, xspace}
\usepackage[lined, boxed,linesnumbered, ruled]{algorithm2e}
\usepackage{caption, subcaption}

\newtheorem{theorem}{Theorem}[section]
\newtheorem{lemma}[theorem]{Lemma}

\theoremstyle{definition}
\newtheorem{definition}[theorem]{Definition}

\usepackage{fancyhdr}
\newcommand{\braces}[1]{\ensuremath{\left\lbrace#1 \right\rbrace}}

\newcommand{\clusterPair}[3][]{\ensuremath{\left(#2,#3\right)^{#1}}}

\newcommand{\clusterSequence}[3]{\ensuremath{#1(#2,#3)}}

\newcommand{\clusterSet}[1]{\ensuremath{\clusterSetSymbol^{#1}}}
\newcommand{\clusterSetOfCluster}[3][]{\ensuremath{\clusterSet{#2}_{#1}\hspace{-2pt}\left(#3\right)}}
\newcommand{\clusterSetSymbol}{\mathcal{S}}

\newcommand{\clusterWalk}[4][]{\ensuremath{#2\hspace{-2.10pt}\left(#3,#4\right)^{#1}}}

\newcommand{\code}[1]{\textsf{#1}}

\newcommand{\conditionalSet}[2]{\ensuremath{ \left\lbrace #1\ \mid  #2\right\rbrace  }}

\newcommand{\disk}[1]{\ensuremath{\diskSymbol^{#1}}}

\newcommand{\diskSymbol}{\ensuremath{\mathcal{D}}}

\newcommand{\falling}[1]{\ensuremath{\underline{#1}}}

\newcommand{\graded}[1]{{\bf #1}}

\newcommand{\List}[3]{\ensuremath{\left(\listKernel{#1}{#2}{#3}\right)}}
\newcommand{\listKernel}[3]{\ensuremath{#1_{{#2}},\ldots,#1_{{#3}}}}

\newcommand{\naturals}{\ensuremath{\mathbb{N}}}

\newcommand{\nSet}[1]{\ensuremath{[#1]}}
\newcommand{\nSetOpen}[1]{\ensuremath{[#1)}}
\newcommand{\nOpenSet}[1]{\ensuremath{[#1)}}
\newcommand{\numberOfStates}[1]{\ensuremath{|#1|}}
\newcommand{\numberOfTransfers}[1]{\ensuremath{L\left(#1\right)}}
\newcommand{\nZeroSet}[1]{\ensuremath{\nSet{#1}_0}}
\newcommand{\nZeroSetOpen}[1]{\ensuremath{\nSetOpen{#1}_0}}

\newcommand{\parentheses}[1]{\ensuremath{\left(#1\right)}}
\newcommand{\Path}{\ensuremath{\pi}}

\newcommand{\posts}{\ensuremath{\mathfrak{P}}}

\newcommand{\pReflect}[4][]{\ensuremath{\reflectiveMap[#1]{#3}{#4}\hspace{-2pt}\left(#2\right)}}
\newcommand{\pStateGraph}[2][]{\ensuremath{\stateGraphSymbol}^{#2}_{#1}}

\newcommand{\pTranslate}[2]{\ensuremath{\pTranslativeMapArg{#2}{#1}}}
\newcommand{\pTranslativeMap}[1]{\ensuremath{\pTranslativeMapSymbol_{#1}}}
\newcommand{\pTranslativeMapArg}[2]{\ensuremath{\pTranslativeMap{#1}\parentheses{#2}}}
\newcommand{\pTranslativeMapSymbol}{\ensuremath{T}}

\newcommand{\reflect}[4][]{\ensuremath{\reflectiveMap[#1]{#2}{#3}\hspace{-2pt}\left(#4\right)}}
\newcommand{\reflectiveMap}[3][]{\ensuremath{\reflectiveMapNotation^{#1}_{#2\mid #3}}}
\newcommand{\reflectiveMapNotation}{\ensuremath{R}}

\newcommand{\sequenceDecomposition}[2]{\ensuremath{[#1]^{#2}}}

\newcommand{\stateGraphSymbol}{\ensuremath{H}}

\newcommand{\transferLength}[1]{\ensuremath{\numberOfTransfers{#1}}}

\newcommand{\tgec}[1]{\ensuremath{\transitionGraphEdgeCountSymbol\parentheses{#1}}}
\newcommand{\transitionGraphEdgeCountSymbol}{\eta}

\title{On the Parallel Tower of Hanoi Puzzle: Acyclicity and a Conditional Triangle Inequality}

\author{Andrey Rukhin}

\begin{document}
	\maketitle
\begin{abstract}
A parallel variant of the Tower of Hanoi Puzzle is described herein. Within this parallel context, two theorems on minimal walks in the state space of configurations, along with their constructive proofs, are provided. These proofs are used to describe a {\sl denoising method}: a method for identifying and eliminating sub-optimal transfers within an arbitrary, valid sequence of disk configurations (as per the rules of the Puzzle). We discuss potential applications of this method to hierarchical reinforcement learning.
\end{abstract}

\section{Introduction}
\subsection{Problem Description}
	The Tower of Hanoi Puzzle consists of $n \ (n\in \naturals_0)$ annular disks (no two disks of equal radius) and $p$ posts $(p\geq
2)$ attached to a fixed base.  The puzzle begins with all of the disks stacked on a single post with no larger disk being
stacked atop a smaller disk.  The goal of the puzzle is to transfer the initial ``tower'' of disks to another post
while adhering to the following rules:
	\begin{enumerate}[i.]
	        \item one disk is transferred from the top of one stack to the top of another (possibly empty) stack at each stage of the puzzle, and
	        \item a larger disk cannot be transferred to a post occupied by a smaller disk.
	\end{enumerate}
	See \cite{HinzBook} for a complete summary and history of the Puzzle.
	One prevailing question that has been studied over the past decades is summarized as follows: assuming feasibility, how does one perform the task in the fewest number of disk transfers?   Recent progress has resolved this question in the case of $p=4$ (\cite{Bousch}),  and this approach has been applied in \cite{Grosu} to improve the asymptotic bounds for the cases where $p\geq 5$).   
	
	This article considers the same question in a generalization of the classic Puzzle where there are  one or more ``towers" of disks. We will introduce $t$ towers $(t\in \naturals)$ on $p$ posts where  $p \geq t+2$.   Each tower of disks is assigned its own color,  and each tower has $n$ disks of sizes $1,\ldots,n$ where $i<j$ implies that  disk $i$ is larger than  disk $j$. Similar to above, we require that no disk (of any color) is atop any smaller or equal-sized disk (of any color). We will define herein the rules of the parallel puzzle in an analogous manner: 
	\begin{enumerate}[i.]
	        \item one or more one disks (of any color) are transferred---each from the top its own stack---to the top of one or more (possibly empty) stacks at each stage of the puzzle, where no two transfers share the same destination stack, and  
	        \item  a larger disk cannot be transferred to a post occupied by a smaller or equal-sized disk.
	\end{enumerate}
	
	  For other parallel variants of the Tower of Hanoi Puzzle (on a single tower of disks, but where more than one disk can be transferred at each stage) see \cite{WU1992241}  and \cite{LU19953}.   
	  
	  Within  our parallel context, we take up the question of finding walks of minimal length within our parallel context, and we establish two results---with constructive proofs---that prescribe necessary properties that minimal-length walks connecting configurations must possess (the single-tower version of these results were proved in \cite{Rukhin}); these constructive proofs are used to construct a {\sl denoising} algorithm for eliminating extraneous/sub-optimal disk transfers from an arbitrary sequence of moves within the parallel puzzle (e.g., a rollout within a reinforcement learning procedure within the context of the Tower of Hanoi Puzzle;  e.g, see \cite{langleyLTS}, \cite{neurIPS},   \cite{BackwardForward}). In the context of learning the minimal paths from an arbitrary configuration to a prescribed destination configuration, the denoising method may accelerate learning by being applied after one or more rollouts are performed: the quality  values of an action obtained after any rollout may be updated and improved with the quality values of actions obtained from the denoised rollout. See Figure \ref{figure::ToHQLearning}.
	  
	   	\begin{figure}[h!]
		\centering	
		\scalebox{.2}{\includegraphics{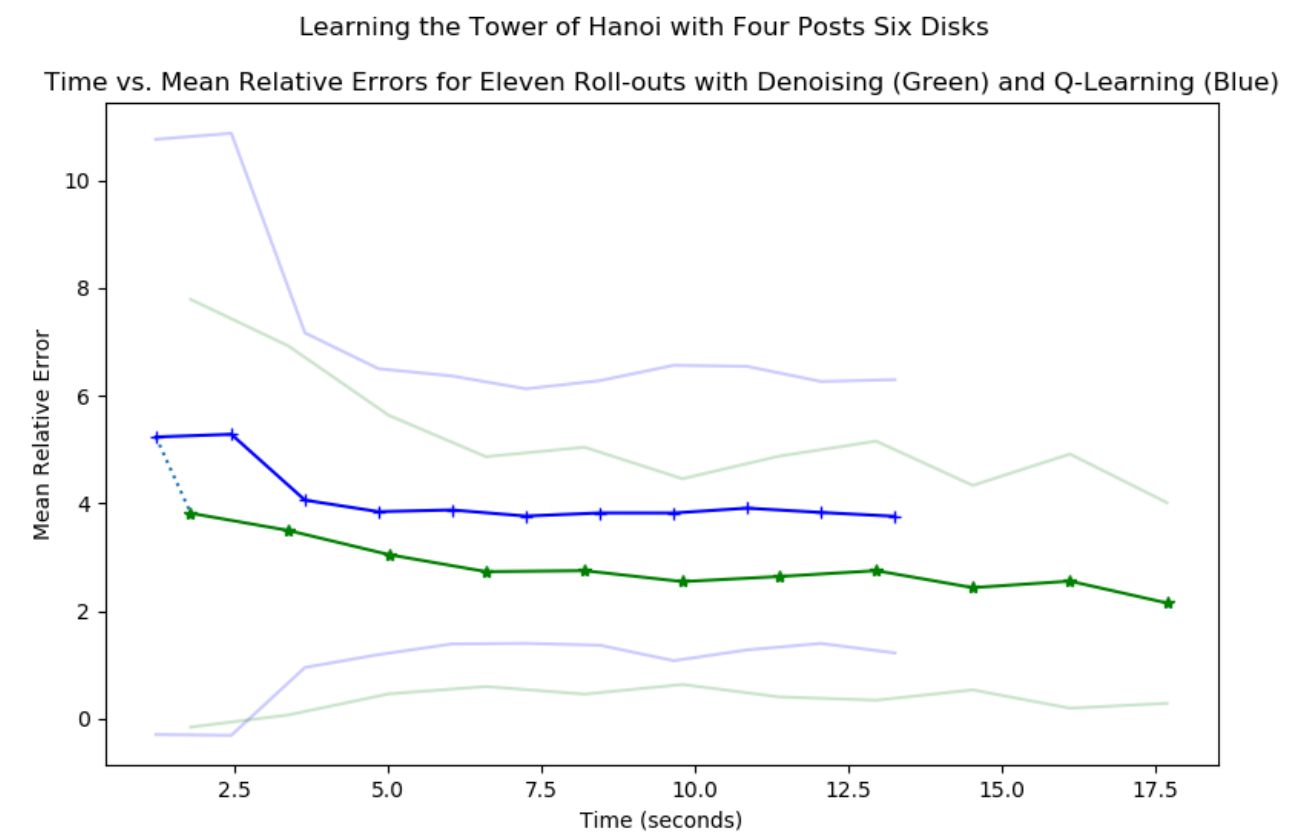}}
		\caption{Comparison of Time vs. Mean Squared Error of a two learning agents (in the single-tower Puzzle on four posts and six disks) after individual rollouts: one with the denoising method incorporated in the learning episodes (green), and one without (blue).  The mean-squared errors at time $\approx 12.7$ seconds yield the following statistics: the two-sample $T$ statistic for comparing the mean squared errors of the learning agents with and without denoising ($1,000$ samples each) is $9.13$ with $\nu = 1992.4$ degrees of freedom. \label{figure::ToHQLearning}}
	\end{figure}
	
		\subsection{Article Summary}
	 In Section \ref{section::setup}, we encapsulate the Parallel Tower of Hanoi Puzzle in a  finite metric space on graphs, words, and sets of words (\emph{clusters}).

	 In Section \ref{section::results}, we establish two results that prescribe necessary properties that minimal-length walks connecting configurations must possess.

	  The proofs of these results are used to describe the {\sl denoising} method (summarized in Section \ref{section::denoising}).

\subsection{Notation}

The domain of all variables is the set of  non-negative integers.  Throughout, we use $n$, $p$, and $t$  to denote the number of disks, posts, and towers, respectively.

For any fixed positive integer $m$, we let \nSet{m} denote the set of the first $m$ positive integers.  We 
write \nZeroSet{m} to denote the set $\nSet{m} \cup \{0\}.$ Furthermore, we'll write $\nSetOpen{m}$ and $\nZeroSetOpen{m}$ to denote the sets $\nSet{m-1}$ and $\nZeroSet{m-1}$, respectively. 
We will define the set of $p$ posts to be \posts. 

Assume $\List{X}{i}{i+k}$ is a
subsequence of $k+1$ contiguous elements within the sequence $\List{X}{1}{m}$ (e.g., a subsequence of contiguous symbols in a string, a
subsequence of  adjacent disk configurations). Should the case arise where $k < 0$,  we will assume that the
subsequence \List{X}{i}{i+k} is the empty
sequence (e.g, the empty string $\epsilon$, the empty subsequence).

For a fixed $t$, we will write $\graded{1}^t =\List{1}{}{} \in \naturals^{t}$, and for $j\in \naturals_0$, we will write $j\cdot\graded{1}^t = \List{j}{}{} \in \naturals_0^t$. Context permitting, we will omit the superscript.

Assume that $\graded{i}, \graded{j} \in \naturals_0^t$ where $\graded{i} = \parentheses{i_u}_{u\in \nSet{t}}$ and  $\graded{j} = \parentheses{j_u}_{u\in \nSet{t}}$. We will say that $\graded{i} \preceq \graded{j}$ (equivalently, $\graded{j} \succeq \graded{i}$) if and only if $i_u \leq j_u$ for each $u\in \nSet{t}$. We will also write $\graded{i} \prec \graded{j}$  (equivalently, $\graded{j} \succ \graded{i}$)  if the inequality $i_u < j_u$ (equivalently, $i_u>j_u$) holds for at least one  index $u$. 

We assume that each of the $t$ towers is assigned its own {\sl color}. We will write $\disk{u,j}$ to denote the $j$-th largest disk of color $u$ ($j\in \nSet{n}, u \in \nSet{t}$).  We will represent a configuration of $nt$ disks as $t\times n$ matrix over the set of posts $\posts$:  the matrix 
\[
\begin{bmatrix}
    a_{1,1}  & \dots  & a_{1,n} \\
    \vdots  & \ddots & \vdots \\
    a_{t,1}  & \dots  &a_{t,n}
\end{bmatrix}
\] encodes a configuration of $t$ towers, each with $n$ disks,  on $p$ posts where the disk $\disk{u,j}$ occupies the post $a_{u,j}$; the largest disks of each color occupy the posts in the in the first (left-most) column, and the post configurations of the smaller disks are listed by column  in decreasing size from left to  right.

\section{The Parallel Tower of Hanoi Problem}\label{section::setup}

In this section, we will establish the mathematical machinery to analyze the state space of disk configurations in a hierarchical manner: we will partition the state space of disk configurations into sets (called {\sl clusters}) that are defined by the arrangement of a prescribed subset of the largest disks.  Afterwards, as was originally done in \cite{Scorer}, we will  represent the parallel state space of configurations in the context of graphs.  We will define the formalism to  analyze walks within these graphs, as well as {\sl cluster walks}: for a pair $(\alpha, \beta)$ of arbitrary configurations within the state space, this machinery will allow us to decide in which subsets of configurations (i.e., the corresponding subsets of vertices in the state graph) that minimal walks connecting $\alpha$ and $\beta$ must be contained. 

\begin{definition}[Parallel Clusters/Cluster Sets]\label{definition::pcluster}
   Assume that $\graded{g} = \List{g}{1}{t} \in \nZeroSet{n}^{t}$.
	\begin{enumerate}
		\item A {\sl cluster (of grading $\graded{g}$)} is a set of disk configurations in which, for each $u \in \nSet{t}$,  the $g_u$  disks
$\disk{u,1},\ldots,\disk{u,g_u}$ are in a fixed configuration. 


We  denote the set of all clusters of \graded{g} on $p$ posts as
$\clusterSet{t, \graded{g}}_p$. 
We  also define the cluster of grading $0\cdot \graded{1}$ to be the set of all disk configurations, and we  denote this set as $\clusterSet{t,n}_{p}$.
		\item Assume the cluster $A \in \clusterSet{t, \graded{g}}_p$. For all $\graded{h} \succeq \graded{g}$, we define the set of clusters
		\[
		 \clusterSetOfCluster[p]{t, \graded{h}}{A} =
\conditionalSet{B}{ B \in \clusterSet{t, \graded{h}}_p,   B \subseteq A}.
		\]  
	\end{enumerate}   
	
	Context permitting, we will omit  the superscript $t$ and subscript $p$ whenever possible.


\end{definition}

%
For $j\in \nSet{n}$, there are $p$ choices to place $\disk{1,j}$, $p-1$ choices to place $\disk{2,j}, \ldots, p-t+1$ choices to place $\disk{t,j}$.  Thus, there are $p^{\falling{t}}$ choices to place the $t$ disks of  index $j$.  Consequently, induction yields  $\parentheses{p^{\falling{t}}}^{j}$ clusters of grading $j\cdot \graded{1}$.

\subsection{EREW-Adjacency (Exclusive Read, Exclusive Write Adjacency) }

In the classical puzzle, two configurations were adjacent if and only if they differed by the valid transfer of a single disk.  We now extend this definition into the parallel context as follows.

\begin{definition}[Exclusive Read/Write Adjacency]

Let the configuration $\alpha = \parentheses{a_{u,y}} \in \clusterSet{t,n}_p$.

A configuration $\beta = \parentheses{a_{u,y}} \in \clusterSet{t,n}_p$ is {\sl EREW-adjacent to $\alpha$} if and only if

the following conditions are satisfied:
\begin{enumerate}\label{definition::adjacency}
	\item\label{nontriviality} {\sl  Non-triviality}:  there exists $u\in \nSet{t}$ and $j_u\in \nSet{n}$ where $a_{u,j_u} \neq b_{u,j_u}$ (at least one disk has transferred), and

	\item\label{stackReadWrite} {\sl  Stack Read/Write}:  if $a_{u,j_u} \neq b_{u,j_u}$, then,  for $v\in \nSet{t}$ and $y> j_u$, the inequalities $a_{v,y} \neq a_{u,j_u}$ and $b_{v,y} \neq b_{u,j_u}$ hold.

\end{enumerate}

\end{definition}

One consequence of this definition is the {\sl exclusive read/write} property for the parallel Puzzle: if $a_{u,j_u} \neq b_{u,j_u}$ and $a_{v,j_v} \neq b_{v,j_v}$ where $u\neq v$, then $a_{u,j_u} \neq a_{v,j_v}$ and  $b_{u,j_u} \neq b_{v,j_v}$. 

Equipped with this definition of adjacency, we will define the {\sl state graph} of the parallel Puzzle as was done in \cite{Scorer} for the single-tower case (see Figure \ref{figure::stateGraphEx}).

\begin{definition}\label{definition::pStateGraph}
	 We  define the {\sl state graph} of the parallel Puzzle to be the pair $\parentheses{\clusterSet{t,n}_p, E}$ where \[
	E = \braces{ \braces{\alpha,\beta} \mid \alpha \textmd{ and } \beta \textmd{ are {\sl EREW}-adjacent configurations}  },
\] and we will denote this graph as $\pStateGraph[p]{t,n}$.
\end{definition}

	\begin{figure}[h!]
		\centering	
		\tdplotsetmaincoords{75}{72}

\resizebox{2in}{!}{\begin{tikzpicture}[tdplot_main_coords]

		\node (00) at (-1.61043535006171, -0.929785282870616, -0.657457478565264) {};
		\node (01) at (-0.777102016728374, -1.02601032773555, -0.113126424613447) {};
		\node (02) at (-1.27710201672837, -0.159984923951115, -0.113126424613447) {};
		\node (03) at (-0.777102016728374, -0.448660058545928, -0.929623005541173) {};
		\node (10) at (0.500000000000000, 1.18599525168667, -0.113126424613447) {};
		\node (11) at (0.000000000000000, 1.85957056574123, -0.657457478565264) {};
		\node (12) at (-0.500000000000000, 1.18599525168667, -0.113126424613447) {};
		\node (13) at (0.000000000000000, 0.897320117091856, -0.929623005541173) {};
		\node (20) at (1.27710201672837, -0.159984923951116, -0.113126424613447) {};
		\node (21) at (0.777102016728374, -1.02601032773555, -0.113126424613447) {};
		\node (22) at (1.61043535006171, -0.929785282870617, -0.657457478565264) {};
		\node (23) at (0.777102016728374, -0.448660058545928, -0.929623005541173) {};
		\node (30) at (0.500000000000000, 0.288675134594813, 1.15587585476807) {};
		\node (31) at (-8.32667268468867e-17, -0.577350269189626, 1.15587585476807) {};
		\node (32) at (-0.500000000000000, 0.288675134594813, 1.15587585476807) {};
		\node (33) at (0.000000000000000, 0.000000000000000, 1.97237243569579) {};


		\draw[fill=lightgray,opacity=0.9] (00.center) -- (01.center) -- (02.center)--cycle;
		
		\draw[fill=lightgray,opacity=0.9] (00.center) -- (02.center) -- (03.center)--cycle;
		\draw[fill=lightgray,opacity=0.9] (00.center) -- (01.center) -- (03.center)--cycle;
		\draw[fill=lightgray,opacity=0.9] (01.center) -- (02.center) -- (03.center)--cycle;
		
		\draw[fill=lightgray,opacity=0.9] (01.center) -- (21.center) -- (31.center)--cycle;
		\draw[fill=lightgray,opacity=0.9] (02.center) -- (12.center) -- (32.center)--cycle;
		\draw[fill=lightgray,opacity=0.9] (03.center) -- (13.center) -- (23.center)--cycle;
		\draw[fill=lightgray,opacity=0.9] (10.center) -- (11.center) -- (12.center)--cycle;
		\draw[fill=lightgray,opacity=0.9] (10.center) -- (11.center) -- (13.center)--cycle;

		\draw[fill=lightgray,opacity=0.9] (11.center) -- (12.center) -- (13.center)--cycle;
		
		\draw[fill=lightgray,opacity=0.9] (10.center) -- (12.center) -- (13.center)--cycle;
		\shadedraw [ball color= lightgray, line width=.25pt] (12) circle (0.04cm);
		\draw[fill=lightgray,opacity=0.9] (10.center) -- (20.center) -- (30.center)--cycle;
		
		\draw[fill=lightgray,opacity=0.9] (01.center) -- (21.center) -- (23.center) -- (03.center)--cycle;
		\draw[fill=lightgray,opacity=0.9] (20.center) -- (21.center) -- (23.center)--cycle;
		
		\draw[fill=lightgray,opacity=0.9] (21.center) -- (22.center) -- (23.center)--cycle;
		\draw[fill=lightgray,opacity=0.9] (20.center) -- (22.center) -- (23.center)--cycle;
		\draw[fill=lightgray,opacity=0.9] (20.center) -- (21.center) -- (22.center)--cycle;
		
		\shadedraw [ball color= lightgray, line width=.25pt] (32) circle (0.04cm);
		\draw[fill=lightgray,opacity=0.9] (30.center) -- (32.center) -- (33.center)--cycle;
		\draw[fill=lightgray,opacity=0.9] (31.center) -- (32.center) -- (33.center)--cycle;
		
		\draw[fill=lightgray,opacity=0.9] (30.center) -- (31.center) -- (32.center)--cycle;
		\draw[fill=lightgray,opacity=0.9] (30.center) -- (31.center) -- (33.center)--cycle;

	\shadedraw [ball color= lightgray, line width=.25pt] (02) circle (0.04cm);

		
		\draw[fill=lightgray,opacity=0.9] (13.center) -- (23.center) -- (20.center) -- (10.center)--cycle;
		\draw[fill=lightgray,opacity=0.9] (20.center) -- (21.center) -- (31.center) -- (30.center)--cycle;

		\shadedraw [ball color= white, line width=.25pt] (00) circle (0.04cm);
		\shadedraw [ball color= gray, line width=.25pt] (01) circle (0.04cm);
		
		\shadedraw [ball color= gray, line width=.25pt] (03) circle (0.04cm);
		\shadedraw [ball color= gray, line width=.25pt] (10) circle (0.04cm);
		\shadedraw [ball color= lightgray!75!black, line width=.25pt] (11) circle (0.04cm);
	
		\shadedraw [ball color= gray, line width=.25pt] (13) circle (0.04cm);
		\shadedraw [ball color= gray, line width=.25pt] (20) circle (0.04cm);
		\shadedraw [ball color= gray, line width=.25pt] (21) circle (0.04cm);
		\shadedraw [ball color= lightgray!50!black, line width=.25pt] (22) circle (0.04cm);
		\shadedraw [ball color= gray, line width=.25pt] (23) circle (0.04cm);

		\shadedraw [ball color= gray, line width=.25pt] (30) circle (0.04cm);
		\shadedraw [ball color= gray, line width=.25pt] (31) circle (0.04cm);
		\shadedraw [ball color= black, line width=.25pt] (33) circle (0.04cm);

	 \end{tikzpicture}
}
			\tdplotsetmaincoords{75}{72}
\resizebox{2in}{!}{
	\begin{tikzpicture}[tdplot_main_coords, line join=round]

		\node (0x1) at (-0.449788251933767, -0.837035637521917, 0.0204988603092420) {};
		\node (0x1L) at (-0.499788251933767, -0.808168124062436, 0.102148518402015) {};
		\node (0x1R) at (-0.399788251933767, -0.865903150981398, -0.0611507977835306) {};

		\node (0x3) at (-0.449788251933767, -0.259685368332291, -0.795997720618484) {};
		\node (0x3L) at (-0.399788251933767, -0.346287908710735, -0.795997720618484) {};
		\node (0x3R) at (-0.499788251933767, -0.173082827953847, -0.795997720618484) {};

		\node (0x2) at (-0.949788251933766, 0.0289897662625216, 0.0204988603092424) {};
		\node (0x2L) at (-0.949788251933766, 0.0867247931814842, -0.0611507977835302) {};
		\node (0x2R) at (-0.949788251933766, -0.0287452606564410, 0.102148518402015) {};

		\node (1x2) at (-0.500000000000000, 0.808045871259395, 0.0204988603092420) {};
		\node (1x2L) at (-0.450000000000000, 0.836913384718877, 0.102148518402015) {};
		\node (1x2R) at (-0.550000000000000, 0.779178357799914, -0.0611507977835306) {};

		\node (1x3) at (0.000000000000000, 0.519370736664583, -0.795997720618484) {};
		\node (1x3L) at (-0.100000000000000, 0.519370736664583, -0.795997720618484) {};
		\node (1x3R) at (0.100000000000000, 0.519370736664583, -0.795997720618484) {};

		\node (1x0) at (0.500000000000000, 0.808045871259396, 0.0204988603092424) {};
		\node (1x0L) at (0.550000000000000, 0.779178357799914, -0.0611507977835302) {};
		\node (1x0R) at (0.450000000000000, 0.836913384718877, 0.102148518402015) {};

		\node (2x0) at (0.949788251933766, 0.0289897662625212, 0.0204988603092420) {};
		\node (2x0L) at (0.949788251933766, -0.0287452606564414, 0.102148518402015) {};
		\node (2x0R) at (0.949788251933766, 0.0867247931814838, -0.0611507977835306) {};

		\node (2x3) at (0.449788251933766, -0.259685368332292, -0.795997720618484) {};
		\node (2x3L) at (0.499788251933767, -0.173082827953848, -0.795997720618484) {};
		\node (2x3R) at (0.399788251933766, -0.346287908710735, -0.795997720618484) {};

		\node (2x1) at (0.449788251933766, -0.837035637521917, 0.0204988603092424) {};
		\node (2x1L) at (0.399788251933766, -0.865903150981398, -0.0611507977835302) {};
		\node (2x1R) at (0.499788251933766, -0.808168124062436, 0.102148518402015) {};

		\node (3x1) at (-1.94289029309402e-16, -0.577350269189626, 0.755000000000000) {};
		\node (3x1L) at (0.0999999999999998, -0.577350269189626, 0.755000000000000) {};
		\node (3x1R) at (-0.100000000000000, -0.577350269189626, 0.755000000000000) {};

		\node (3x2) at (-0.500000000000000, 0.288675134594813, 0.755000000000000) {};
		\node (3x2L) at (-0.550000000000000, 0.202072594216369, 0.755000000000000) {};
		\node (3x2R) at (-0.450000000000000, 0.375277674973257, 0.755000000000000) {};

		\node (3x0) at (0.500000000000000, 0.288675134594813, 0.755000000000000) {};
		\node (3x0L) at (0.450000000000000, 0.375277674973256, 0.755000000000000) {};
		\node (3x0R) at (0.550000000000000, 0.202072594216369, 0.755000000000000) {};

		\draw[fill=silver, opacity=.9] (0x2.center)--(0x1.center)--(0x3.center)--cycle;

		\draw[fill=silver, opacity=.9] (1x2.center)--(1x3.center)--(1x0.center)--cycle;

		\draw[fill=silver, opacity=.9] (3x2.center)--(1x2.center)--(0x2.center)--cycle;


		\shadedraw [ball color= blue!20!white, line width=.25pt] (0x1L) circle (0.04cm);
		\shadedraw [ball color= ao(english)!40!white, line width=.25pt] (0x1R) circle (0.04cm);

		\shadedraw [ball color= blue!20!white, line width=.25pt] (0x3L) circle (0.04cm);
		\shadedraw [ball color= ao(english), line width=.25pt] (0x3R) circle (0.04cm);

		\shadedraw [ball color= blue!20!white, line width=.25pt] (0x2L) circle (0.04cm);
		\shadedraw [ball color= ao(english)!60!white, line width=.25pt] (0x2R) circle (0.04cm);
		\shadedraw [ball color= ao(english)!60!white, line width=.25pt] (1x2R) circle (0.04cm);
		\shadedraw [ball color= blue!40!white, line width=.25pt] (1x3L) circle (0.04cm);
		\draw[fill=silver, opacity=.9] (0x2.center)--(0x3.center)--(1x3.center)--(1x2.center)--cycle;
		
		
		\shadedraw [ball color=  blue!40!white, line width=.25pt] (1x2L) circle (0.04cm);

	
		\shadedraw [ball color= ao(english), line width=.25pt] (1x3R) circle (0.04cm);

		\shadedraw [ball color=  blue!40!white, line width=.25pt] (1x0L) circle (0.04cm);
		\shadedraw [ball color= green!20!white, line width=.25pt] (1x0R) circle (0.04cm);

		\draw[fill=silver, opacity=.9] (3x2.center)--(3x0.center)--(1x0.center)--(1x2.center)--cycle;
		\draw[fill=silver, opacity=.9] (3x1.center)--(3x2.center)--(0x2.center)--(0x1.center)--cycle;
		\shadedraw [ball color= blue, line width=.25pt] (3x2L) circle (0.04cm);
		\draw[fill=silver, opacity=.9] (2x0.center)--(2x1.center)--(3x1.center)--(3x0.center)--cycle;			
		\draw[fill=silver, opacity=.9] (3x1.center)--(3x2.center)--(3x0.center)--cycle;
		\draw[fill=silver, opacity=.9] (2x0.center)--(2x1.center)--(2x3.center)--cycle;
		\draw[fill=silver, opacity=.9] (3x0.center)--(2x0.center)--(1x0.center)--cycle;
		\draw[fill=silver, opacity=.9] (3x1.center)--(2x1.center)--(0x1.center)--cycle;
		\draw[fill=silver, opacity=.9] (1x0.center)--(2x0.center)--(2x3.center)--(1x3.center)--cycle;
		\draw[fill=silver, opacity=.9] (0x3.center)--(2x3.center)--(2x1.center)--(0x1.center)--cycle;ama
		\shadedraw [ball color= blue!60!white, line width=.25pt] (2x0L) circle (0.04cm);
		\shadedraw [ball color= green!20!white, line width=.25pt] (2x0R) circle (0.04cm);

		\shadedraw [ball color= blue!60!white, line width=.25pt] (2x3L) circle (0.04cm);
		\shadedraw [ball color= ao(english), line width=.25pt] (2x3R) circle (0.04cm);

		\shadedraw [ball color= blue!60!white, line width=.25pt] (2x1L) circle (0.04cm);
		\shadedraw [ball color= ao(english)!40!white, line width=.25pt] (2x1R) circle (0.04cm);

		\shadedraw [ball color= ao(english)!40!white, line width=.25pt] (3x1R) circle (0.04cm);
		\shadedraw [ball color= blue, line width=.25pt] (3x1L) circle (0.04cm);

		
		\shadedraw [ball color= ao(english)!60!white, line width=.25pt] (3x2R) circle (0.04cm);

		\shadedraw [ball color= blue, line width=.25pt] (3x0L) circle (0.04cm);
		\shadedraw [ball color= green!20!white, line width=.25pt] (3x0R) circle (0.04cm);

	 \end{tikzpicture}
}
		\caption{Illustrated Examples of the graphs $H^{1,2}_4$ and $H^{2,1}_4$; only the edges on the hull of $H^{2,1}_4 \equiv K_{12}$ are presented.}\label{figure::stateGraphEx}
	\end{figure}
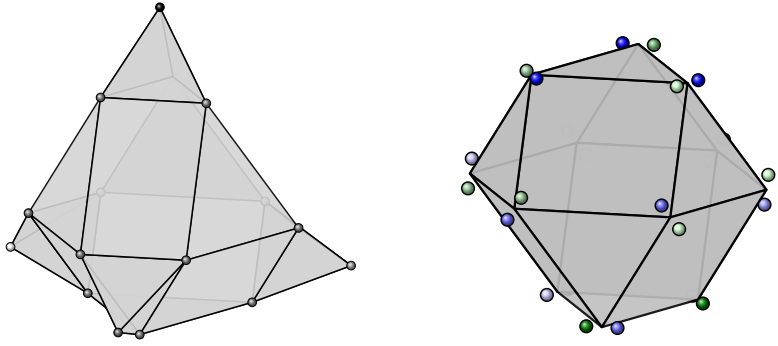

We now define a {\sl cluster walk}, and we will equip walks in the parallel state space with a measure in the parallel state space that naturally aligns with the counting measure of the classic Puzzle. To this end,
we will introduce an auxiliary graph object for the parallel Puzzle.

\begin{definition}\label{definition::transitionGraph}
   Let configurations $\alpha$ and $\beta$ be EREW-adjacent.  We define the {\sl transition graph} of the adjacent pair to be the directed, edge-labelled, graph $\parentheses{\posts, \graded{j}}$  on  $\posts$  with edge-labels over $\nSet{t} \times \nSet{n}$: the edge $(q, r)$ is in the edge set of the transition graph with  label  $\parentheses{u,j_u}$ if and only if the disk $\disk{u,j_u}$ transfers from post $q$ to post $r$. 
   
   The edge set $\graded{j}$ is defined to be the {\sl transfer vector} of the EREW pair $\parentheses{\alpha, \beta}$. We will say that the configurations are {\sl $\graded{j}$-adjacent}, and express this relation with the notation $\clusterPair[\graded{j}]{\alpha}{\beta}$. 
   
   We denote the number of edges in the transition graph of the pair $\clusterPair[\graded{j}]{\alpha}{\beta}$ to be $\tgec{\alpha,\beta}$. If $\alpha = \beta$, then we define $\tgec{\alpha,\beta} = 0$.

   
\end{definition}

As per the rules of the parallel Puzzle, the in-degree and out-degree of any vertex in the transition graph is at most one.

We will now define adjaceny for clusters as follows. 


\begin{definition}[Parallel Cluster Pair]\label{definition::pClusterAdjacency}
	Let clusters $A,B\in \clusterSet{t,\graded{g}}_p$ for some fixed grading $\graded{g} \in \nSet{n}^t$ where $A\neq B$.  The clusters $A$ and $B$ are {\sl EREW-adjacent} if and only if there
are configurations $\alpha\in\clusterSetOfCluster{n}{A}$ and $\beta\in\clusterSetOfCluster{n}{B}$  where $\clusterPair[\graded{k}]{\alpha}{\beta}$ for some transfer vector $\graded{k}$. 

	Let $\graded{j}$ denote the subset\footnote{As $A\neq B$, we are assured that $\graded{j}\neq \emptyset$.} of $\graded{k}$ where
	\[
	\graded{j} = \braces{ (q,r) \mid \textmd{the edge label of } (q,r) \textmd{ is } (u,j_u) \textmd{ where } j_u \leq g }.
	\]  We will say that $A$ and $B$ form a {\sl $\graded{j}$-adjacent cluster pair (with transfer vector $\graded{j}$)}, denoted by $\clusterPair[\graded{j}]{A}{B}$.

\end{definition}

\begin{definition}[Parallel Cluster Walk]\label{definition::pClusterWalk}
	Let $A,B\in\clusterSet{t,\graded{g}}_p$ for some fixed grading $\graded{g} \in \nSet{n}^t$.  A {\sl (cluster) walk (of grading $\graded{g}$)} connecting $A$ and $B$, denoted as 
	\[
		\clusterSequence{\Path}{A}{B} = \List{A}{1}{w},
	\]  is a sequence of clusters of grading $\graded{g}$ (a {\sl $\graded{g}$-walk})
	 that satisfies the following properties: $A = A_1$, $B = A_w$, and, for $u \in \nOpenSet{w}$,  either
$A_u = A_{u+1}$ or $\clusterPair[\graded{j}]{A_u}{A_{u+1}}$ for some transfer vector $\graded{j}$.

	

%

\end{definition}

We will also define a {\sl valid} sequence of configurations $\List{\alpha}{1}{w}$ analogously (we allow repeated configurations in a valid sequence of configurations).

We will define both the {\sl sequence length} and {\sl transfer length} of configuration sequences.    If $\Path = \List{\alpha}{1}{w}$ is a configuration sequence, then its sequence length is defined to be $\numberOfStates{\Path} = w$ (the number of configurations in $\Path$), and its transfer length is defined to be $\transferLength{\Path} = \sum_{1\leq v < w} \tgec{\alpha_v,\alpha_{v+1}}$ (the sum of the number of edges in all transition graphs of $\Path$).

We will also say that a sequence is {\sl contained} in some cluster $A$ if and only if each configuration in the sequence is an element of $\clusterSetOfCluster[p]{t,n}{A}$.

In order to establish results on the transfer length of configuration sequences, we will formally prescribe disk transfers as mappings on such sequences.

\begin{definition}[Translative and Reflective Mappings]

 Let the disk configuration $\alpha =  \left(a_{u,y}\right)\in  \clusterSet{t,n}_p$, and  let  $g\in \nSet{n}$.
  
  \begin{enumerate}[i.]
\item  Let $\graded{g} = g\cdot\graded{1} \in \nSet{n}^t$. Let the cluster $B\in \clusterSet{t,\graded{g}}_p$ where
 the disk $\disk{u,y}$ occupies the post $b_{u,y}$ for $u\in \nSet{t}$ and $y\in \nSet{g}$. 
 The translative map $\pTranslativeMap{B}: \clusterSet{t,n}_p \to \clusterSet{t,n}_p(B)$ maps $\alpha$
 to the configuration \pTranslativeMapArg{B}{\alpha}, where
\[
\pTranslativeMapArg{B}{\alpha} = 
\begin{bmatrix}
    b_{1,1}  & \dots & b_{1,g} &a_{1,g+1} & \dots & a_{1,n} \\
    \vdots  & \ddots  & \vdots& \vdots & \ddots & \vdots \\
    b_{t,1}  & \dots  & b_{t,g} & a_{t,g+1} & \dots  &a_{t,n}
\end{bmatrix}.
\] 

\item   Let  $q,r\in \posts$. 
The reflective map  $\reflectiveMap[g]{q}{r}: \clusterSet{t,n}_p \to \clusterSet{t,n}_p$ maps $\alpha$  to the configuration $\reflectiveMap[g]{q}{r}\parentheses{\alpha} = \left(b_{v,y}\right)$, where,   
 for $v\in \nSet{t}$ and $y\geq g$, 
\begin{enumerate}
	\item if $a_{v,y} = q$, then  $b_{v,y} = r$;
	\item if $a_{v,y} = r$, then $ b_{v,y} = q$;
	\item if $a_{v,y} \notin\{r,q\}$, then $ b_{v,y} = a_{v,y}$.
\end{enumerate}
\end{enumerate}
\end{definition}

We will now identify a class of mappings that for which the images of valid configuration sequences are themselves valid.


\begin{lemma}[Mappings on Configuration Sequences]\label{lemma::clusterMappings} Let $\Path = \List{\alpha}{1}{w}$ be a valid sequence of configurations.  Also, let $g\in \nSet{n}$,  and let $\graded{g} = g\cdot\graded{1} \in \nSet{n}^t$. 

\begin{enumerate}
	\item \label{TranslativeMappingProof}
For any cluster $B \in \clusterSet{t, \graded{g}}_p$, the sequence
	$
	\pTranslate{\Path}{B} = \parentheses{\pTranslate{\alpha_v}{B}}_{v\in \nSet{w}}
	$ is a  valid sequence of configurations contained in $B$.
	
	\item If there exists a cluster $B \in \clusterSet{t, \graded{g}-\graded{1}}_p$ that contains \Path, then, for $q,r\in \posts$,  the sequence 
	$\reflect[g]{\Path}{q}{r} = \parentheses{\reflect[g]{\alpha_v}{q}{r}}_{v\in \nSet{w}}$ is a  valid sequence of configurations contained in $B$.
\end{enumerate}
\end{lemma}

\begin{proof} Assume the hypotheses and notation in the statement of the lemma.  For each $v\in \nSet{w}$, let $\alpha_{v}  = \parentheses{a^{(v)}_{u,y}}$.
	\begin{enumerate}

		\item   Assume that the cluster $B$ is such that
 the disk $\disk{u,y}$ occupies the post $b_{u,y}$ for $u\in \nSet{t}$ and $y\in \nSet{i}$. 
 
 If $\alpha_v = \alpha_{v+1}$, then we have the equality $\pTranslate{\alpha_v}{B} = \pTranslate{\alpha_{v+1}}{B}$.
		
   		If $\clusterPair[\graded{j}]{\alpha_v}{\alpha_{v+1}}$  for some transfer vector $\graded{j}$, then write $\pTranslate{\alpha_v}{B}  = \parentheses{b^{(v)}_{u,y}}$ where
		\[
		b^{(v)}_{u,y} = 
			\begin{cases}
				a^{(v)}_{u,y} & y > g\\
				b_{u,y} & y \leq g. 
			\end{cases}
		\]
		The claim is that either $\pTranslate{\alpha_v}{B} = \pTranslate{\alpha_{v+1}}{B}$ or $\clusterPair[\graded{i}]{\pTranslate{\alpha_v}{B} }{\pTranslate{\alpha_{v+1}}{B} }$ for some transfer vector $\graded{i}$ where $\graded{i} \subseteq \graded{j}$.
		
		Let \graded{i} denote the subset of $\graded{j}$ where
		\[
		\graded{i} = \braces{ (q,r) \mid \textmd{the edge label of } (q,r) \textmd{ is } (u,j_u) \textmd{ where } j_u > g }.
		\]  If $\graded{i}$ is empty, then we have the equality $\pTranslate{\alpha_v}{B} = \pTranslate{\alpha_{v+1}}{B}$. Otherwise, if $a^{(v)}_{u,j_u} \neq a^{(v+1)}_{u,j_u}$ where $j_u > g$, then $a^{(v)}_{w,y} \neq a^{(v)}_{w,j_u}$ and $a^{(v+1)}_{w,y} \neq a^{(v+1)}_{u,j_u}$ for $w\in \nSet{t}$ and $y > j_u$ (by the definition of the \graded{j}-adacency of $\alpha_v$ and $\alpha_{v+1}$).  As $j_u > g$, we have that $b^{(v)}_{u,j_u} =a^{(v)}_{u,j_u}$ and $b^{(v+1)}_{u,j_u} =a^{(v+1)}_{u,j_u}$; furthermore, as the equalities $b^{(v)}_{w,y} = a^{(v)}_{w,y}$ and $b^{(v+1)}_{w,y} = a^{(v+1)}_{w,y}$ hold for $w>j_u$, it follows that $b^{(v)}_{w,y} \neq b^{(v)}_{w,j_u}$ and $b^{(v+1)}_{w,y} \neq b^{(v+1)}_{u,j_u}$.  Thus, the configurations $\pTranslate{\alpha_v}{B} $ and $\pTranslate{\alpha_{v+1}}{B}$ form an \graded{i}-adjacent pair $\clusterPair[\graded{i}]{\pTranslate{\alpha_v}{B} }{\pTranslate{\alpha_{v+1}}{B} }$.
		
%
%
			

%
%
		\item If $\alpha_v = \alpha_{v+1}$, then we have the equality $\reflect[g]{\alpha_v}{q}{r} = \reflect[i]{\alpha_{v+1}}{q}{r}$.  
		
		If $q=r$, then $\pReflect[g]{\Path}{q}{r} = \Path$. Thus, we will assume that $q\neq r$.
		
		If $\clusterPair[\graded{j}]{\alpha_v}{\alpha_{v+1}}$ for some transfer vector \graded{j}, then, under the assumption that $\sequenceDecomposition{\Path}{\graded{g} - \graded{1}} = (B)$, it must be the case that each edge in $\graded{j}$ has an edge label $(u,j_u)$ where $j_u\geq g$.  Let $\pReflect[g]{\alpha_v}{q}{r} = \parentheses{b^{(v)}_{u,y}}$, where, for $y \geq g$, we define
		\[
		b^{(v)}_{u,y} = 
		\begin{cases}
			r &  a^{(v)}_{u,y} = q,  \\
			q &  a^{(v)}_{u,y} = r,  \\
			a^{(v)}_{u,y} & a^{(v)}_{u,y} \notin \{q,r\} .
		\end{cases}
		\]  We will show that $\clusterPair[\graded{j}]{\pReflect[i]{\alpha_v}{q}{r}}{\pReflect[i]{\alpha_{v+1}}{q}{r}}$ forms a cluster pair. 
		\begin{enumerate}
			\item If $a^{(v)}_{u,j_u} \neq a^{(v+1)}_{u, j_u}$, then: 
				\begin{enumerate}
					\item if $a^{(v)}_{u,j_u} = q$ and $a^{(v+1)}_{u,j_u} = r$, then $b^{(v)}_{u,j_u} = r$, and $b^{(v+1)}_{u,j_u}=q$; an analogous argument holds for when $a^{(v)}_{u,j_u} = r$ and $a^{(v+1)}_{u,j_u} =q $;
					\item  if $a^{(v)}_{u,j_u} = q$ and $a^{(v+1)}_{u,j_u} \notin  \{q,r\}$, then $b^{(v)}_{u,j_u}  = r$, and $b^{(v+1)}_{u,j_u}=a^{(v+1)}_{u,j_u} \notin \{q,r\}$;  analogous arguments hold for when  $a^{(v)}_{u,j_u}= r$ and $a^{(v+1)}_{u,j_u} \notin  \{q,r\}$, and when $a^{(v)}_{u,j_u}\notin \{q,r\}$ and $a^{(v+1)}_{u,j_u}\in  \{q,r\}$;
					\item  if $a^{(v)}_{u,j_u} \notin \{q,r\}$ and $a^{(v+1)}_{u,j_u} \notin  \{q,r\}$, then $b^{(v)}_{u,j_u} = a^{(v)}_{u,j_u}$, and $b^{(v+1)}_{u,j_u}=a^{(v+1)}_{u,j_u}$.
				\end{enumerate}  In all cases, the inequality $b^{(v)}_{u,j_u}\neq b^{(v+1)}_{u,j_u}$ holds, and the non-triviality condition \ref{definition::adjacency}.\ref{nontriviality} is met.

			\item\label{ReflectiveMapStackReadWriteProof} If $a^{(v)}_{u,j_u} \neq a^{(v+1)}_{u,j_u}$, then  $a^{(v)}_{s,y} \neq a^{(v)}_{u,j_u} $ and  $a^{(v+1)}_{s,y} \neq a^{(v+1)}_{u,j_u} $ for $s\in \nSet{t}$ and $y>j_u$, as per condition \ref{definition::adjacency}.\ref{stackReadWrite}. \begin{enumerate}
					\item if the post $a^{(v)}_{u,j_u} = q$ and the post $a^{(v+1)}_{u,j_u} = r$, then the post $b^{(v)}_{u,j_u} = r$, and the post $b^{(v+1)}_{u,j_u} = q$;   furthermore, for $y> j_u$,  the post $b^{(v)}_{s,y} \neq r$ (otherwise, the post $a^{(v)}_{s,y} = q$), and the post $b^{(v+1)}_{s,y}\neq q$ (otherwise, the post $a^{(v+1)}_{s,y} = r$); an analogous argument holds for when $a^{(v)}_{u,j_u} = r$ and $a^{(v+1)}_{u,j_u} =q$;
					\item  if the post $a^{(v)}_{u,j_u} = q$ and the post $a^{(v+1)}_{u,j_u} \notin  \{q,r\}$, then the post $b^{(v)}_{u,j_u} = r$ and the post $b^{(v+1)}_{u,j_u} = a^{(v+1)}_{u,j_u}\notin  \{q,r\}$.  Furthermore, the post $b^{(v)}_{s,y} \neq r$ (otherwise, the post $a^{(v)}_{s,y} = q$), and the post 
						\[
							b^{(v+1)}_{s,y}= 
							\begin{cases}
								q &a^{(v+1)}_{s,y} = r\\
								r & a^{(v+1)}_{s,y} = q\\ 
								a^{(v+1)}_{s,y} & a^{(v+1)}_{s,y}\notin \{q,r\};
							\end{cases}
						\]  in each case, as the post $b^{(v+1)}_{u,j_u} = a^{(v+1)}_{u,j_u} \notin \{q,r\}$, we have the inequality $b^{(v+1)}_{s,y}  \neq b^{(v+1)}_{u,j_u}$.  Analogous arguments hold for when  the post $a^{(v)}_{u,j_u} = r$ and $a^{(v+1)}_{u,j_u} \notin  \{q,r\}$, and when the post $a^{(v)}_{u,j_u}\notin \{q,r\}$  and the post $a^{(v+1)}_{u,j_u} \in  \{q,r\}$;
					\item  if $a^{(v)}_{u,j_u} \notin \{q,r\}$ and $a^{(v+1)}_{u,j_u} \notin  \{q,r\}$, then $b^{(v)}_{u,j_u} = a^{(v)}_{u,j_u}$, and $b^{(v+1)}_{u,j_u}=a^{(v+1)}_{u,j_u}$; as before, whether $a^{(v)}_{s,y}\in \{q,r\}$ or not (and whether  $a^{(v+1)}_{s,y} \in \{q,r\}$ or not), we have the inequalities $b^{(v)}_{s,y} \neq b^{(v)}_{u, j_u}$ and $b^{(v+1)}_{s,y} \neq b^{(v+1)}_{u, j_u}$.
				\end{enumerate}
		\end{enumerate} In all cases, the stack read/write condition \ref{definition::adjacency}.\ref{stackReadWrite} is met.
	\end{enumerate}
\end{proof}

\section{Results}\label{section::results}

An elementary result of graph theory states that any minimal path in a graph connecting a pair of vertices must be acyclic. 
The first theorem extends the notion of {\sl acyclicity} of minimal walks into the context of parallel cluster walks.

\begin{theorem}\label{theorem::1}
	Let $g\in \nSet{n}$, let $\graded{g} = g\cdot \graded{1} \in \nSet{n}^t$,  and let $A \in \clusterSet{t,\graded{g}}_p$.  If the configurations $\alpha,\beta \in \clusterSetOfCluster[p]{t,n}{A}$, then every minimal configuration
sequence $\clusterWalk{\mu}{\alpha}{\beta} \subseteq \clusterSetOfCluster[p]{t,n}{A}$.
\end{theorem}

\begin{proof}\label{proof::1}
	Assume the notation within the hypotheses of the theorem statement. 
	We will show that for any walk $\clusterSequence{\nu}{\alpha}{\beta}$ where $\nu \nsubseteq \clusterSetOfCluster{n}{A}$,
there exists a walk
\clusterSequence{\mu}{\alpha}{\beta} with the property that $\mu \subseteq \clusterSetOfCluster[p]{t,n}{A}$ and
$\numberOfTransfers{\nu}>\numberOfTransfers{\mu}$.
	
	As $\clusterSequence{\nu}{\alpha}{\beta} \nsubseteq \clusterSetOfCluster{n}{A}$, there exists a unique sequence of clusters $\List{B}{1}{m}$ of grading \graded{g} where $B_1 = B_m = A$, and  $\clusterPair[\graded{i}_w]{B_w}{B_{w+1}}$ for some transfer vector $\graded{i}_w$ for each $w\in \nOpenSet{m}$.

	Thus, we write $\nu$ as \[
	 \nu = \sum_{B_w}\nu_w,
	\]  where $\nu_w \subseteq \clusterSetOfCluster[p]{t,n}{B_w}$, and 
	\[
	\nu_w = \List{\alpha}{w,1}{w,l(w)}.
	\]

	For $w \in \nOpenSet{m}$, the $\graded{j}_w$-adjacent configurations $\clusterPair[\graded{j}_w]{\alpha_{w,l(w)}}{\alpha_{w+1,1}}$ satisfy the properties that $\alpha_{w,l(w)} \in \clusterSetOfCluster[p]{t,n}{B_w}$ and $\alpha_{w+1,1} \in \clusterSetOfCluster[p]{t,n}{B_{w+1}}$ where $\clusterPair[\graded{i}_w]{B_w}{B_{w+1}}$.
	
	By Definition \ref{definition::pClusterAdjacency}, the transfer vector
	\[
	\graded{i}_w = \braces{ (q,r) \mid \textmd{the edge label of } (q,r) \textmd{ is } (u,j_u) \textmd{ where } j_u \leq g } \neq \emptyset;
	\]  thus, as per the proof in  \ref{lemma::clusterMappings}.\ref{TranslativeMappingProof}, we have the following inequality on the edge counts of the translation graphs:
	\[
	\tgec{
		\pTranslate{\alpha_{w,l(w)}}{A}, 
		\pTranslate{\alpha_{w+1,1}}{A}
		} < 
	\tgec{
		\alpha_{w,l(w)},
		\alpha_{w+1,1}
		}
	\] (as the transfer of any disk of index less than or equal to $g$ is removed from the set $\graded{i}_w$).

	Thus, the walk $\mu = \pTranslate{\nu}{A}$ is contained in $\clusterSetOfCluster[p]{t,n}{A}$ as per Lemma \ref{lemma::clusterMappings},  and $\numberOfTransfers{\mu} <
\numberOfTransfers{\nu}$.
	
\end{proof}

The second theorem yields a conditional triangle inequality within the context of parallel clusters.

\begin{theorem}\label{theorem::2}  Fix $g\in \nSet{n}$,  let $\graded{g} = g\cdot \graded{1} \in \nSet{n}^t$, and let  $\graded{g-1} = (g-1)\cdot \graded{1} \in \nSet{n}^t$.  Fix $A_0 \in \clusterSet{t, \graded{g-1}}_p$ where the disk \disk{u,y} occupies the post $a_{u,y}$ for $u\in \nSet{t}$ and $y \in \nOpenSet{g}$.

Let $a,b,c\in \posts$ be pairwise unequal post values.  Fix $v\in \nSet{t}$, and let $\graded{g}_v =  \List{g-1+\delta}{v,1}{v,t}$ where $\delta_{v,u} = [u=v]$ for each $u\in \nSet{t}$.  Let the clusters $A, B$ and $C$ be elements of $\clusterSetOfCluster[p]{t,\graded{g}_v}{A_0}$ where the disk $\disk{v,g}$ occupies the post $a, b$ and $c$, respectively.

Let the configuration $\alpha\in\clusterSetOfCluster[p]{t,n}{A}$ be such that  the identity
\[
\reflect[g]{b}{c}{\alpha} = \alpha
\] holds, and let $\beta$ be a cluster where $\beta\in\clusterSetOfCluster[p]{t,n}{B}$. Let $\clusterSequence{\nu}{\alpha}{\beta}$ be a configuration sequence connecting $\alpha$ and $\beta$ where \[
	\nu = \nu_A + \nu_C + \nu_B,
	\] and the parts $\nu_A$, $\nu_B$, and $\nu_C$ are contained in the clusters $A$, $B$, and $C$, respectively.
	
	Then, there exists 
 a  configuration sequence
$\clusterSequence{\mu}{\alpha}{\beta}$ where \[\clusterSequence{\mu}{\alpha}{\beta} = \mu_A + \mu_B,\] the part $\mu_A $ is contained in $A$, the part  $ \mu_B $ is contained in $B$,  and the sequence $\clusterSequence{\mu}{\alpha}{\beta}$ satisfies the inequality
\[
\numberOfTransfers{\mu} <
\numberOfTransfers{\nu}.
\]
\end{theorem}


\begin{proof}\label{proof::2}

	Assume the notation and hypotheses within the theorem statement.  We note that, by applying Theorem \ref{theorem::1} to walks connecting configurations in $\clusterSetOfCluster[p]{t,n}{A_0}$, we can  assume that the  configuration sequence $\nu$ is contained in the cluster $A_0$.  Furthermore, by applying Theorem  \ref{theorem::1} to walks connecting configurations within the clusters $A, B$ and $C$, we assume that we can uniquely express the walk $\nu$ as
	\[
	\nu = \nu_A + \nu_C + \nu_B,
	\] where $\nu_A, \nu_B$, and $\nu_C$  satisfy the Theorem statement assumptions.

Assume that the subwalk 
$$\nu_A =
\List{\alpha}{1}{x},$$ 
 the subwalk 
 $$\nu_C = \List{\gamma}{1}{z},$$
 and the subwalk  
  $$\nu_B =
\List{\beta}{1}{y}.$$

The following conditions hold within the images of $\nu_{A}$ and $\nu_C$ under the reflective map $\reflectiveMap[g]{b}{c}$:
\begin{enumerate}[i.]
	\item The cluster $\reflect[g]{\alpha_1}{b}{c} = \alpha_1$ by assumption.
	\item The image $ \reflect[g]{\nu_C}{b}{c}$ is contained in $\clusterSetOfCluster[p]{t,n}{B}$.
	\item The configurations $\gamma_z$ and $\beta_1$ are \graded{j}-adjacent for some transfer vector \graded{j}, and the pair $(c,b) \in \graded{j}$ with the edge label $(v,g)$. 
		
\end{enumerate}

We will now construct a transfer vector $\graded{i}$ where  $\graded{i}\subsetneq \graded{j}$ and $\clusterPair[\graded{i}]{\reflect[g]{\gamma_z}{b}{c}}{\beta_1}$.

Firstly, we express the configuration $\gamma_{z} = \parentheses{c_{u,y}}$,  the configuration $\beta_1= \parentheses{b_{u,y}}$, and the configuration $\reflect[g]{\gamma_z}{b}{c} = \parentheses{c'_{u,y}}$ for $u\in \nSet{t}$ and $y\in \nSet{n}$. 
 We have that the configurations $\gamma_z, \reflect[g]{\gamma_z}{b}{c}$ and $\beta_1$ are elements of $ \clusterSetOfCluster[p]{t,n}{A_0}$; thus, we have the equalities
	\[
	c_{u,y} = b_{u,y} = c'_{u,y}
	\] for $y<g$. 

We now address the case of static disks: if the post $c_{u,y} = b_{u,y} = q$ for $y\geq g$, then it must be the case that $q\notin \{b,c\}$ (otherwise, the transfer of disk $\disk{v,g}$ is prohibited). Thus, under the assumption that the disk $\disk{u,y}$ statically occupies post $q$ in  \clusterPair[\graded{j}]{\gamma_z}{\beta_1}, we have that $c'_{u,y} = b_{u,y} = q$.

We now address the case of active disks in \clusterPair[\graded{j}]{\gamma_z}{\beta_1}: as the post $c_{v,g} = c$, and the post $b_{v,g} = b$, we have that the post $c'_{v,g} = b_{v,g} =b$, and  the edge $(c,b) \notin \graded{i}$.

Let $(q,r) \in \graded{j}$ with the edge label $(w,j_w)$ where $(q,r) \neq (c,b)$. As per Definition \ref{definition::adjacency}.\ref{stackReadWrite}, it must be the case that $q \neq c$ and $r\neq b$.   Moreover, as $(c,b)\in \graded{j}$ with edge label $(v,g)$, it follows that 
$c_{u,y} \neq c$ and $b_{u,y} \neq b$ for $u\in \nSet{t}$ and $y>g$.

 Under these conditions, we condition by cases as follows. Assume that $h > j_w$ and $u\in \nSet{t}$:
\begin{enumerate}
	\item $q\neq b, r \neq c$: The post $c'_{w,j_w} = c_{w,j_w} = q$. The post $c_{y,h} \neq q$ (as per Definition \ref{definition::adjacency}.\ref{stackReadWrite}), and the post $c'_{u,h} \neq q$ (whether or not $c_{y,h} = b$). The post $b_{u,h} \neq r$; thus, the edge $(q,r) \in \graded{i}$ with with the edge label $(w,j_w)$.
	\item $q = b, r\neq c$:  The post $c'_{w,j_w} = c$, and as the post $c_{u,z} \neq b$, the post $c'_{u,z} \neq c$. The post $b_{u,h} \neq r$; thus, the edge $(c, r) \in \graded{i}$ with with the edge label $(w,j_w)$.
	\item $q\neq b, r = c$: We have that $c'_{w,j_w} = c_{w,j_w} = q$, and $c_{u,z} \neq q$; as $q\notin \{b,c\}$, it follows that $c'_{u,z} \neq q$ (whether or not $c_{u,z} = b$). The post $b_{u,h} \neq c$; thus, the edge $(q,c) \in \graded{i}$ with with the edge label $(w,j_w)$.
	\item $q = b, r = c$: the post $c'_{w,j_w} = c = b_{w,j_w}$; thus, the disk transfer of $\disk{w,j_w}$ is removed, and the edge $(b,c) \notin \graded{i}$.
\end{enumerate}

	Consequently, we have that either 
	$\clusterPair[\graded{i}]{ \reflect[g]{\gamma_z}{b}{c}}{\beta_1}$ or $\reflect[g]{\gamma_z}{b}{c} = \beta_1$, and 
	the inequality
	\[
	\tgec{ \reflect[g]{\gamma_z}{b}{c}, \beta_1} < \tgec{\gamma_z, \beta_1}
	\] holds, as the transfer of the disk $\disk{v,g}$ from post $c$ to post $b$ is not included in \graded{i}.

  Thus, as per Lemma \ref{lemma::clusterMappings}, we define the valid configuration sequence
	\[
	 \mu = \reflect[g]{\nu_A+\nu_C}{b}{c} + \nu_B;
	\] 
	this sequence is contained in $\clusterSetOfCluster[p]{t,n}{A}\cup \clusterSetOfCluster[p]{t,n}{B}$, and  $\numberOfTransfers{\mu}<\numberOfTransfers{\nu}$.

\end{proof}

\newpage

\section{The \code{Denoise} Method}\label{section::denoising}
Let $\Path$ denote an arbitrary configuration sequence.  The constructive methods within the proofs in Section \ref{section::results} can be iteratively applied to $\Path$ to yield a walk $\Path'$ that has the same starting and ending configurations, along with the property that $\transferLength{\Path'} \leq \transferLength{\Path}$.  

	\begin{algorithm}[H]\label{algorithm::denoise}
		\SetKwData{this}{this}
		\SetKwData{child}{child}
		\SetKwData{denoise}{denoise}
			$\Path  = \List{\alpha}{1}{m}$\;
			\For{$g : 1 \to n$}{
				$\graded{g} \gets g\cdot \graded{1}$\;

				\While{Theorem \ref{theorem::1} applies or Theorem \ref{theorem::2} applies}{
					\If{Theorem \ref{theorem::1} applies for $(\alpha_u, \alpha_v) \ (u < v)$}{
						Apply the translative cluster mappings in the proof of Theorem \ref{theorem::1} to remove $\List{\alpha}{u+1}{v}$\;
					}
				
					\If{Theorem \ref{theorem::2} applies for $(\alpha_u,\alpha_v)$, or for $(\alpha_v,\alpha_u)$ in the reversal of the sequence $\ (u < v )$}{
					
						Apply the  reflective cluster mappings in the proof of Theorem \ref{theorem::2} to remove $\List{\alpha}{u+1}{v-1}$\;
						
					}

 				}
			}
			
		\Return{}\;
		\caption{\code{denoise}(\Path)}
	\end{algorithm}

\bibliography{pToHBib}
	\bibliographystyle{plain}

\end{document}